\documentclass[review]{elsarticle}

\usepackage{lineno,hyperref}
\modulolinenumbers[5]

\journal{Journal of \LaTeX\ Templates}

\usepackage{array}

\usepackage{anysize}
\usepackage{fancyheadings}
\usepackage{fancyhdr}
\usepackage[section]{placeins}

\usepackage{float}
\usepackage{amssymb}
\usepackage{amsmath}
\usepackage{amsfonts}
\usepackage{amsthm}
\usepackage{mathrsfs} 
\usepackage{bbm}
\usepackage{bbold}

\usepackage{graphicx}
\usepackage{sidecap}
\usepackage{tikz}
\usetikzlibrary{cd}
\usetikzlibrary{decorations}
\usepackage{tikz-cd}
\usetikzlibrary{decorations.pathmorphing}

\usetikzlibrary{cd}

\usepackage{graphicx}

\usepackage{empheq}

\usepackage{color}

\usepackage{graphicx}

\newtheorem{theorem}{Theorem}[section]
\newtheorem{lemma}{Lemma}[section]
\newtheorem{prop}{Proposition}[section]

\newtheorem{claim}{Claim}[section]

\theoremstyle{definition}
\newtheorem{deff}{Definition}[section]

\theoremstyle{remark}
\newtheorem{remark}{Remark}[section]

\begin{document}

\begin{frontmatter}

\title{Control under constraints for multi-dimensional reaction-diffusion monostable and bistable equations}

\author[UAM,Deusto]{Domènec Ruiz-Balet}

\ead{domenec.ruizi@uam.es}
\corref{mycorrespondingauthor}
\cortext[mycorrespondingauthor]{Corresponding author at: Departamento de Matem\'aticas, Universidad Aut\'onoma de Madrid, 28049 Madrid, Spain.
E-mail address: domenec.ruizi@uam.es (Domènec Ruiz-Balet)}

\author[FAU,Deusto,UAM]{Enrique Zuazua}

\ead{enrique.zuazua@fau.de}

\address[UAM]{Departamento de Matem\'aticas, Universidad Aut\'onoma de Madrid, 28049 Madrid, Spain}
\address[Deusto]{Chair of Computational Mathematics, Fundaci\'on Deusto,    Av. de las Universidades, 24,    48007 Bilbao, Basque Country, Spain}
\address[FAU]{Chair in Applied Analysis,
Alexander von Humboldt-Professorship, 
Department of Mathematics,
Friedrich-Alexander-Universit\"at Erlangen-N\"urnberg,91058 Erlangen, Germany}

\date{\today}

\begin{abstract}
    
Dynamic phenomena in social and biological sciences can often be modeled by 
reaction-diffusion equations. 
When addressing the control from a mathematical viewpoint,
one of the main challenges is that,
because of the intrinsic nature of the models under consideration, 
the solution, typically a proportion or a density function,
needs to preserve given lower and upper bounds
(taking values in $[0,1])$). 
Controlling the system to the desired final configuration then becomes
 complex, and sometimes even impossible.
 In the present work, we analyze the controllability
 to constant steady-states of spatially homogeneous monostable and bistable semilinear
 heat equations, with constraints in the state,
 and using boundary controls.  We prove that controlling the system to a constant steady-state may become impossible when the diffusivity is too small
 due to the existence of barrier functions. 
 We build sophisticated control strategies combining the dissipativity of the system, the existence of traveling waves, and some connectivity of the set of steady-states to ensure controllability whenever it is possible. This connectivity allows building paths that the controlled trajectories can follow, in a long time, with small oscillations, preserving the natural constraints of the system. This kind of strategy was successfully implemented in one-space dimension, where phase plane analysis techniques allowed to decode the nature of the set of steady-states. These techniques fail in the present multi-dimensional setting. We employ a fictitious domain technique, extending the system to a larger ball, and building paths of radially symmetric solution that can then be restricted to the original domain. 
 
 \noindent\textbf{Resumé}
 
De nombreux phénomènes dynamiques en sciences sociales et biologiques peuvent souvent être modélisés par des équations de réaction-diffusion. Lorsque l'on aborde le contrôle de ces équations d'un point de vue mathématique, l'un des principaux défis est que, en raison de la nature intrinsèque des modèles considérés, la solution, généralement une proportion ou une fonction de densité, doit conserver des bornes inférieures et supérieures données (en prenant des valeurs dans [0, 1]). 
Contrôler le système à la configuration finale souhaitée devient alors complexe, et parfois même impossible.
Dans le présent travail, nous analysons la contrôlabilité à des états stationnaires constants pour des équations de la chaleur semi-linéaires spatialement homogènes monostables et bistables, avec des contraintes sur l'état, et en utilisant des contrôles aux bords. 
Nous prouvons que le contrôle du système à un état stable constant peut devenir impossible lorsque la diffusivité est trop faible en raison de l'existence de fonctions de barrière. 
Nous construisons des stratégies de contrôle sophistiquées
combinant la dissipativité du système, l’existence d’ondes progressives et une certaine connectivité des
ensemble d'états stationnaires pour assurer la contrôlabilité chaque fois que cela est possible. 
Cette connectivité permet de construire des chemins que les trajectoires contrôlées peuvent suivre, sur une longue période, avec de petites oscillations, en preservant les contraintes naturelles du système. 
Ce type de stratégie a été mis en œuvre avec succès dans le cas d’une  dimension spatiale, où les techniques d’analyse de plan de phase ont permis de décoder la nature de l’ensemble de états. Ces techniques échouent dans le cadre multidimensionnel actuel. 
Nous utilisons une technique de domaine fictif, étendant le système à une boule plus grande et construisant des chemins de solutions radialement symétriques, qui peut par suite être transféré au domaine d'origine.
 \end{abstract}


\end{frontmatter}

\linenumbers

\section{Introduction}
In this work we consider the boundary controllability of homogeneous monostable and bistable reaction diffusion equations under natural state constraints coming from the physical model.

\subsection{Motivation}
Reaction-diffusion equations frequently appear in nature in a wide variety of phenomena, such as: 
population dynamics and invasion of species (see the pioneering work of Kolmogorov \cite{KOLMOGOROV37}), neuroscience
\cite{JEVANS}, chemical reactions \cite{PERTHAME}, evolutionary game theory \cite{HOFBAUER,Hutson2000}, magnetic systems in material science and
their phase transitions \cite{DEMASI}, linguistics  \cite{VOGL}, etc.

 The nonlinearities of the models that we are going to discuss fall in two important types: monostable and bistable (Figure \ref{nonlinG}).
    \begin{figure}[H]
    \begin{center}

 \includegraphics[scale=0.35]{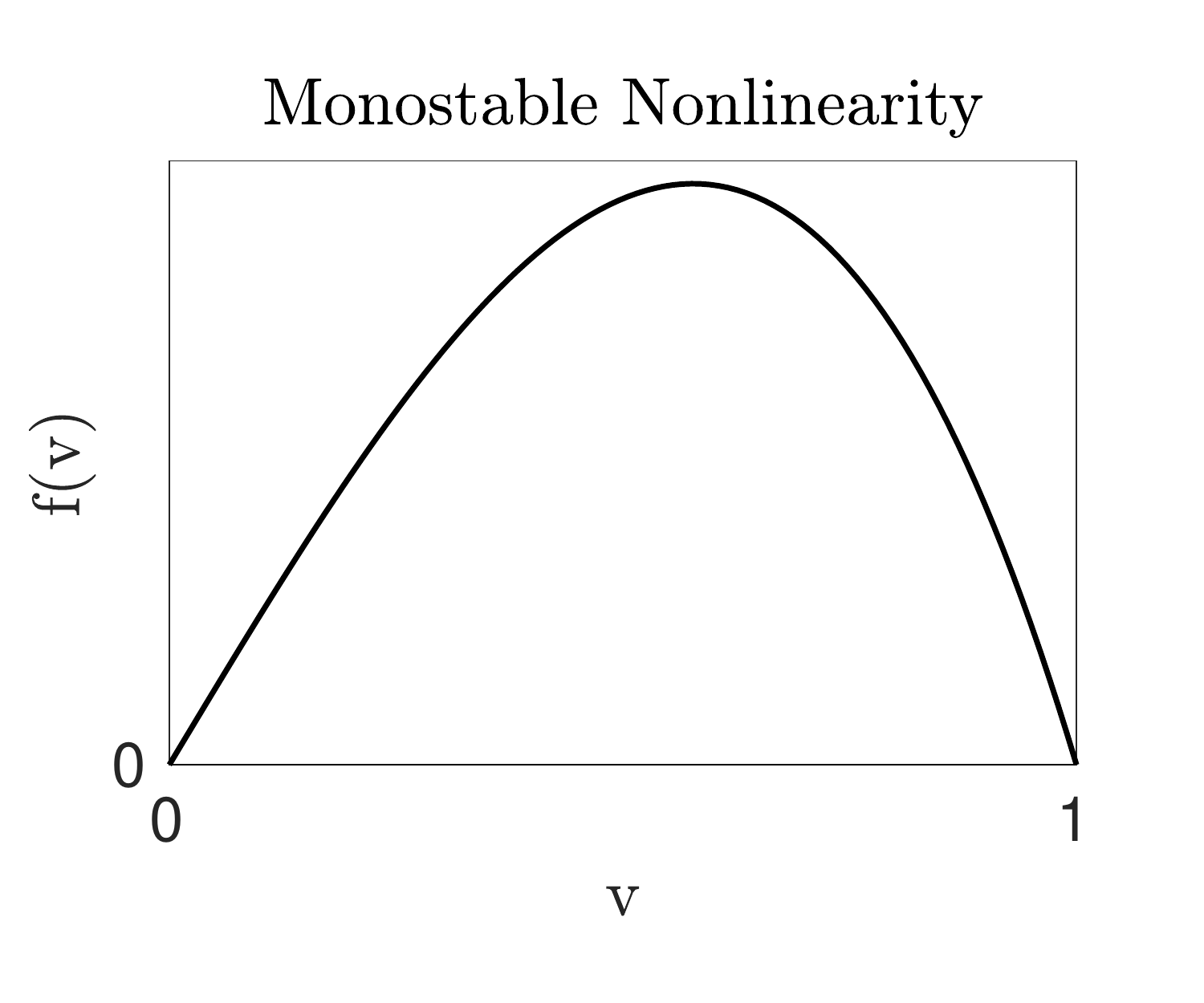}
  \includegraphics[scale=0.35]{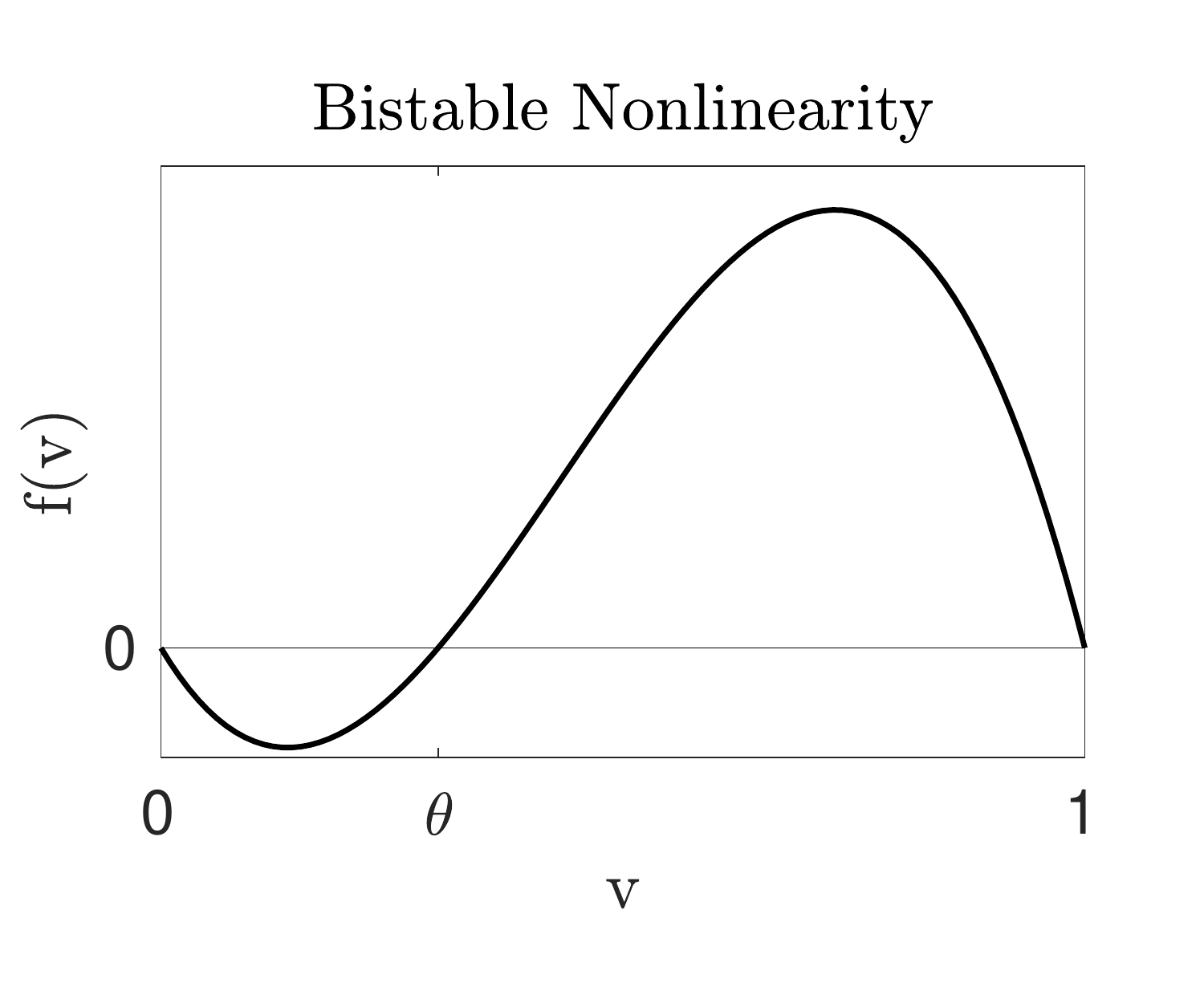}
  \caption{ Monostable nonlinearity (left), bistable nonlinearity (right).}\label{nonlinG}
      \end{center}

\end{figure}

 Among the different control mechanisms, the boundary control is of particular interest since it represents to modify the value in the boundary of a domain whose
accessibility might not be possible. In applications, this quantity is a density, a proportion or a volume fraction. For this reason, the state might be constrained to be non-negative  or even to take values between $0$ and $1$. Obviously, the Dirichlet control needss to fulfill the same restrictions.

However, from the control perspective, the classical methodology, for instance, the control with minimum $L^2$ norm \cite{fursikov1996} (see also \cite{FC-ZZ,lebeau1995}) should not be applied since it can violate the state constraints as we observe in Figure \ref{violation}. The models mentioned above make sense when the state quantities are positive or between $0$ and $1$.

If these constraints are violated during the control process, we lose the physical meaning of the model. The controllability under state constraints of such equations has already been treated in \cite{POUCHOL} for the one-dimensional case.

 \begin{figure}[h]
    \begin{center}
 \includegraphics[scale=0.45]{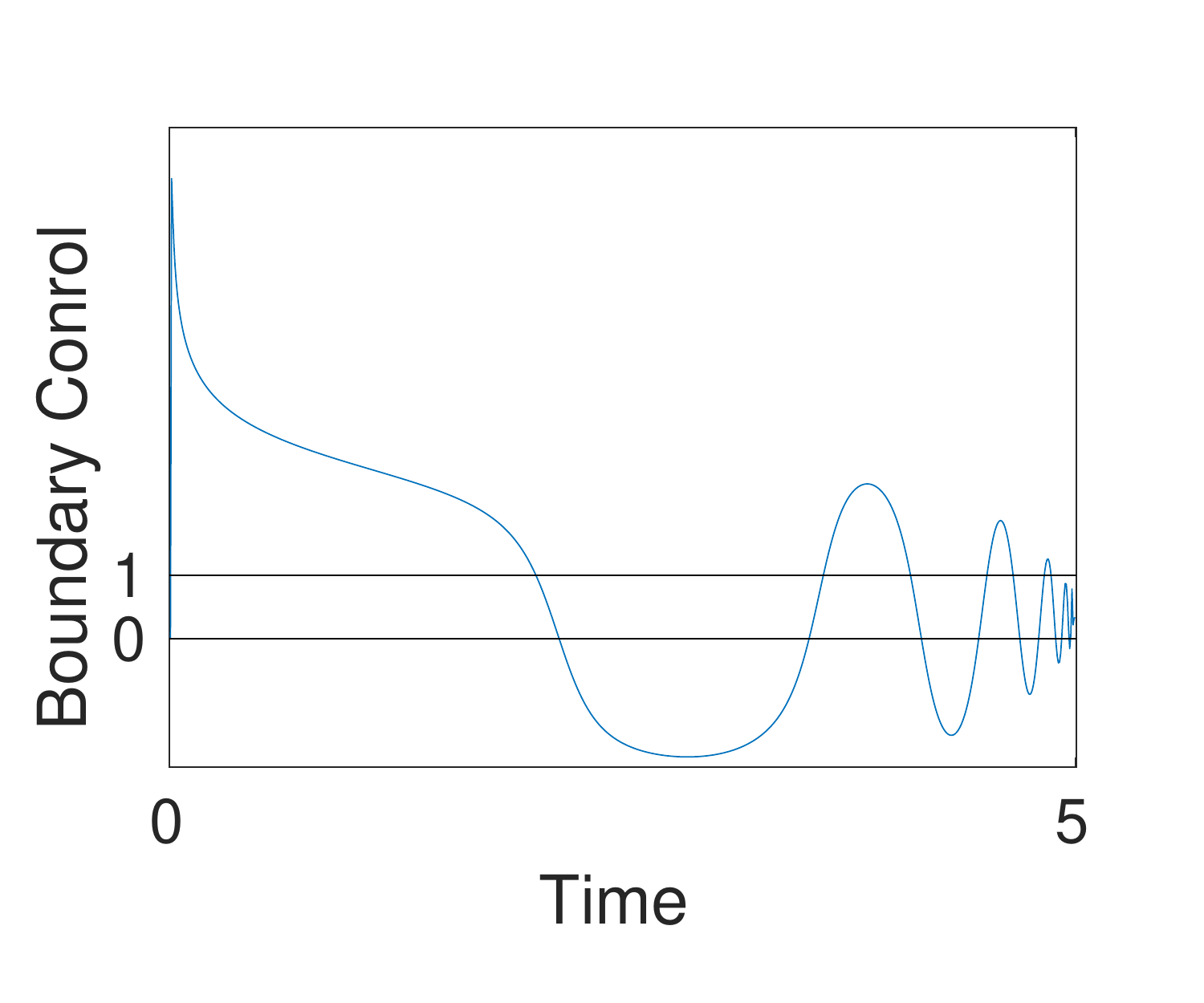}
 \includegraphics[scale=0.45]{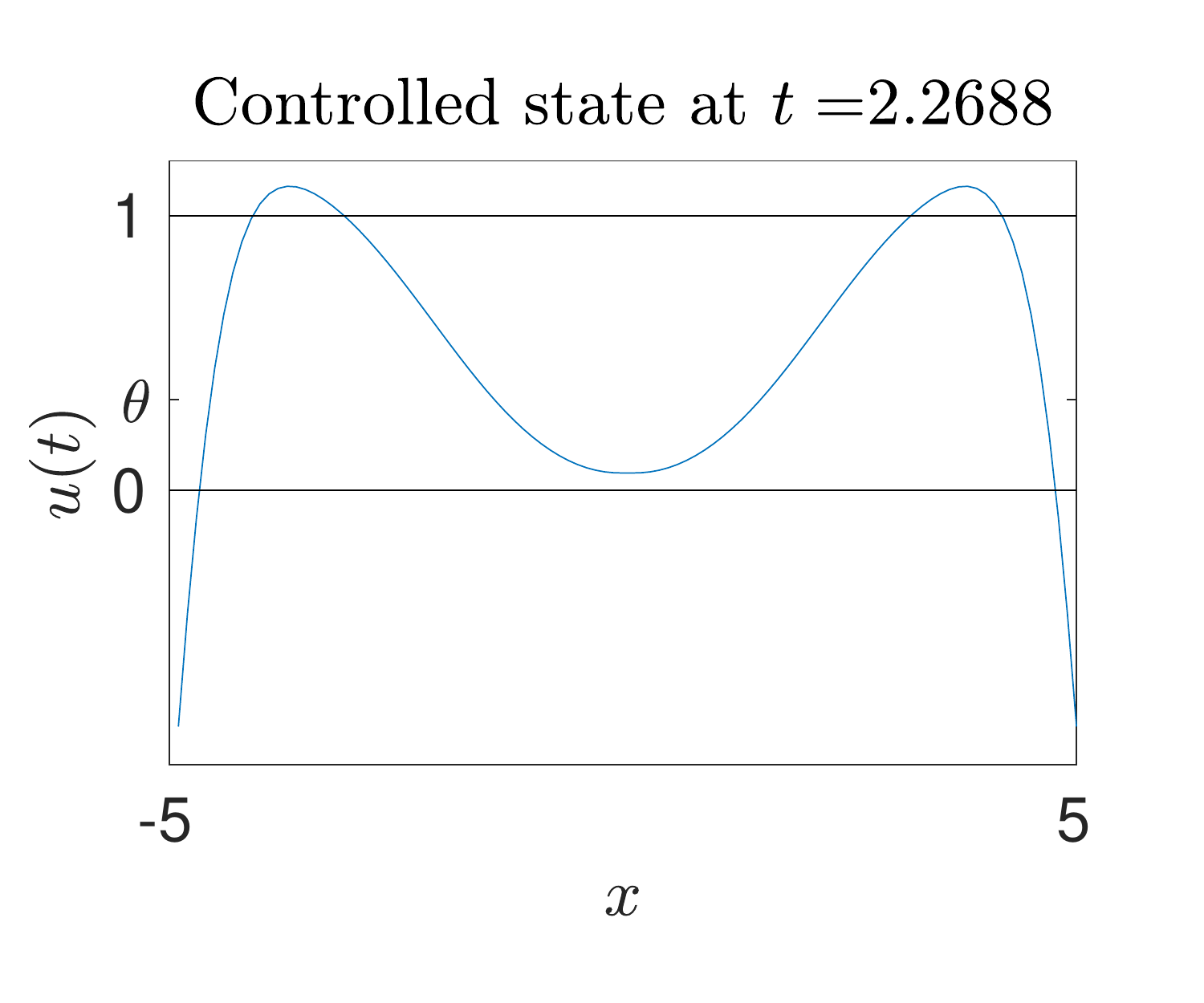}
 \caption{Boundary control of one dimensional bistable equation $u_t-\mu\Delta u=f(u)$ in a bounded domain with diffusivity $\mu=0.75$ and final time $T=5$ from initial state $u_0\equiv 0$ to
 target $w\equiv \theta$. The control violates the constraints imposed by the physical
 nature of the problem.}\label{violation}
  \end{center}
\end{figure}

\subsection{Statement of the problem}

We study the controllability towards certain steady-states under state constraints
of reaction-diffusion equations. Let $\Omega\subset \mathbb{R}^N$ 
be a given bounded $C^2$-regular domain and consider
the following reaction-diffusion equation where we allow ourselves to act with a control function $a$ on the whole boundary:
\begin{equation}\label{maineq}
 \begin{dcases}
  u_t-\mu\Delta u=f(u)&\quad (x,t)\in\Omega\times(0,T],\\
  u(x,t)=a(x,t)&\quad (x,t)\in\partial\Omega\times(0,T],\\
  0\leq u(x,t)\leq 1&\quad (x,t)\in\Omega\times[0,T],
 \end{dcases}
\end{equation}
where $\mu>0$ is the diffusivity constant. We will discuss two types of nonlinearities $f:[0,1]\to\mathbb{R}$, monostable and bistable (see Figure \ref{nonlinG}).

A function $f\in C^1(\mathbb{R})$ is called monostable if:
\begin{align*}
 f(0)=f(1)=0,\quad
 f'(0)>0,\quad f'(1)<0, \quad
 f(s)>0\text{ if } 0<s<1,
\end{align*}
and called bistable if:
\begin{align*}
 &f(0)=f(\theta)=f(1)=0,\quad
 f'(0)<0,\quad f'(1)<0, \quad f'(\theta)>0,\\
 &f(s)<0\text{ if }0<s<\theta, \qquad f(s)>0\text{ if } \theta<s<1.
\end{align*}
For any $T>0$, $u_0\in L^\infty(\Omega;[0,1])$ and $a\in L^\infty(\partial\Omega\times[0,T];[0,1])$, the problem \eqref{maineq} has a unique weak solution $u\in L^\infty (\Omega\times [0,T])\cap C([0,T];H^{-1}(\Omega))$. This follows
by transposition and a fixed point argument (see \cite[Ch.13]{LionsMagenesII}  and \cite{DARIO}). 

In the monostable case typically we require that $u\geq0$
since it models a density of population.

Moreover,  in the bistable case $u$ stands for a proportion of individuals or a concentration of mass,
therefore the model carries the natural constraint:
\begin{equation*}
   0\leq u(x,t)\leq 1.
\end{equation*}
We have to build a control strategy that respects 
this constraint for all time.  However, this is naturally guaranteed by the comparison
principle \cite{PROTTER2012maximum,FIFEcomparison}  because the boundary control fulfills the constraint:
\begin{equation*}
   0\leq a(x,t)\leq 1.
\end{equation*}
It has to be noticed
that this equivalence between state constraints and control constraints
does not hold for any system; this is the case of the wave equation where the maximum principle does not hold (see \cite{DARIOWAVE}).
The solution of \eqref{maineq} is globally defined for any initial data taking values in $u_0(x)\in[0,1]$ since $w\equiv 0$
(respectively $w\equiv 1$) is a subsolution (supersolution).
\begin{deff}
 We say that Equation \eqref{maineq} is controllable to $w\in L^\infty (\Omega,[0,1])$ for an initial datum $u_0\in L^\infty(\Omega,[0,1])$ if there exists some  $T_{u_0}>0$  and a control $a\in L^\infty(\partial\Omega\times[0,T_{u_0}];[0,1])$ such that the solution of \eqref{maineq} with initial data $u_0$ and control $a$ is equal to $w$ at time $T_{u_0}>0$, i.e. $u(T_{u_0};a,u_0)=w$.
\end{deff}
Note that the relevant constraint $0\leq u\leq 1$ on the controlled state is built in the present definition. Moreover, notice that this definition differs from the standard notion of controllability (see \cite{CORON}) in the sense that our time horizon $T$ is not fixed a priory  and depends on the initial datum we are considering.

Our main goal is to discuss the controllability of \eqref{maineq} to the following constant steady-states:
\begin{equation*}
 w\equiv\alpha\qquad \alpha\in\{0,1\},
\end{equation*}
in the monostable case , and 
\begin{equation*}
 w\equiv\alpha\qquad \alpha\in\{0,\theta,1\},
\end{equation*}
in the bistable case.

\subsection{Main features of the problem and main results}
Our first novel result is negative. The admissible controls are between $0$ and $1$. For this reason, when constraints in parabolic models are present, the main cause of lack of controllability is the existence of nontrivial elliptic solutions with boundary values being equal to the limit bounds, i.e. $a=0$ and $a=1$.

For $\mu>0$ small enough the following problem can have a solution

\begin{equation}\label{nontrivial}
 \begin{dcases}
  -\mu\Delta v=f(v)&\quad x\in\Omega,\\
  v=0&\quad x\in\partial\Omega,\\
  0< v< 1&\quad x\in\Omega.
 \end{dcases}
\end{equation} 

\begin{deff}[Barrier]
 A solution of \eqref{nontrivial} is called a barrier.
\end{deff}

Trajectories of \eqref{maineq} corresponding to barriers as initial data (or any initial data above a barrier) cannot be stabilized to $w\equiv 0$. Indeed, this is a consequence of the comparison principle \cite{PROTTER2012maximum,FIFEcomparison}: for any control function $0\leq a(x,t)\leq 1$ and an initial data $u_0$ being a barrier or any function above a barrier, the solution of \eqref{maineq} will be above a barrier.
This leads to a fundamental obstruction for stabilizing around $w\equiv 0$.

Furthermore, a barrier not only prevents to reach the steady-state $w\equiv0$ but also to reach $w\equiv\theta$, since barriers have its maximum value above $\theta$ (see Proposition \ref{MaxPosSol}).

 A nontrivial solution with Dirichlet condition equal to one would have the same effect for approaching $w\equiv 1$ (or $w\equiv \theta$) if we start below this nontrivial solution. However, we will see that such nontrivial solutions do not exist by using a comparison argument with  traveling wave solutions.

 The existence of barriers \eqref{nontrivial} depends basically on a relation between $\mu$, $f$ and $\Omega$, and it has been studied for instance in  \cite{PLLEPSSEQ}.

 A way to study the existence and non-existence of solutions of \eqref{nontrivial} is by means of finding critical points for the energy functional associated to \eqref{nontrivial}:
\begin{equation*}
 J(u)=\int_\Omega \frac{\mu}{2}|\nabla u|^2- F(u) dx,
\end{equation*}
where $F(u(x)):=\int_0^{u(x)}f(s)ds$.  One can see that if $\mu$ is sufficiently big the convex part dominates and $w\equiv 0$ is the unique solution.
However, if $\mu$ is small enough, the term $\int_\Omega F(u)dx$ can be dominant. If the second term is dominant, then there exist a function $v\in H^1_0(\Omega)$, $v\neq 0$ which is a critical point of the functional. 

Observe that the emergence of a barrier occurs in the case of the monostable nonlinearity, since we have that $F(u)\geq 0$. For the bistable case, the same paradigm is found under the assumption that $F(1)>0$. 

In addition, it has to be noticed that the obstruction that a barrier creates depends on being a barrier or above it. Therefore, there might be other initial data for which the target $w\equiv\theta$ can be reached. We will determine the set of initial data for which controllability to $w\equiv\theta$ will hold regardless of $\mu$:  The basin of attraction of $w\equiv 0$, $\mathcal{A}$:
\begin{deff}[Basin of attraction of $w\equiv0$]\label{basin}
We define the basin of attraction of $w\equiv0$ as follows:
    \begin{equation}
        \mathcal{A}:=\left\{u_0\in L^\infty\left(\Omega;[0,1]\right) \text{ such that }\omega_0(u_0)=\{w\equiv0\}\right\} 
    \end{equation}
    where $\omega_0(u_0)$ is the $\omega$-limit by the dynamics \eqref{maineq} with $a\equiv 0$ of $u_0$. Analogously we denote by $\omega_a(u_0)$ the $\omega$-limit by the dynamics \eqref{maineq} with a control function $a\in L^\infty(\partial\Omega\times(0,+\infty);[0,1])$.
\end{deff}

When using a control function taking values in $[0, 1]$, both $w\equiv1$ and $w\equiv0$ are respectively super
and subsolutions, and this implies, by the comparison principle, that, at most, those steady-states can be reached
asymptotically as $t$ tends to infinity, but never in a finite time-horizon

In the bistable case,  nontrivial solutions $w \not\equiv \theta$ may exist with boundary value $a=\theta$. In this case, the simple strategy of setting the boundary control $a(x, t) =\theta$  for large time, stabilizing, and finally applying a local
control does not work. The reason is that these nontrivial steady-state solutions with boundary $\theta$ may attract the dynamics when keeping the boundary control  $a=\theta$, making it impossible to stabilize  the system (with this strategy) into a neighborhood of $w \equiv \theta$. This is why we need to develop more sophisticated control strategies.

The previous discussion points out that at least two regimes can be studied depending on the existence of barriers. Let us introduce the following quantity:
\begin{deff}\label{defpl}
 We define
 \begin{equation*}
  \mu^*(\Omega,f):=\sup\left\{\mu\in\mathbb{R}^+ \text{for which there is a solution of \eqref{nontrivial}}\right\}
 \end{equation*}
 Analogously, we define $\mu_\theta^*(\Omega,f)$ as the supremum value such that a nontrivial solution $w\not\equiv \theta$ of the problem \eqref{nontrivial} exists with Dirichlet boundary conditions equal to $\theta$.
\end{deff}
We will devote more attention to estimates on the thresholds of nontrivial solutions in Section \ref{SMultiplicity}.

The staircase method presented in \cite{DARIO}, roughly speaking, tells us that the problem of controlling to steady-states respecting constraints can be addressed by finding a control such that can steer the system to an admissible continuous path of steady-states connected to the target,  i.e., a connected family of steady-states that take values between $0$ and $1$. The steady-states along the path need to have traces not
touching the extremal values  $0$ and $1$. This leaves  room for the small oscillations needed for the control to drive the system from one steady-state to another.  But this procedure typically requires a long time. The staircase method was used in \cite{POUCHOL} for the one-dimensional bistable semilinear heat equation. We will introduce both results in Section \ref{PREL}.

In the multidimensional case, the construction of the admissible path is done by a restriction argument, considering
a ball containing our domain. The path of steady-states is built using radially symmetric solutions in the ball. This leads, by restriction, to a path of steady-states in the original domain. However, by this argument, the control needs to act all over the boundary of the domain under consideration.

The path constructed exists for every $\mu>0$ and every domain $\Omega$, and it connects $w\equiv 0$ with $w\equiv \theta$. The staircase method allows controlling from any  element of this path to any other other one, in both senses. However, our initial datum, in general, does not belong to the path.  Therefore, in a first control phase, out of the given initial datum, we need to control the system to some point in this path. However, note that, depending of the initial datum of departure, it might be easier to approximate some specific points of this path.

With the construction of the path, we show that nontrivial solutions with boundary value $\theta$ do not constitute a fundamental obstruction to reach $w\equiv\theta$.

Our main results are Theorem \ref{THmono} for monostable equations and Theorem \ref{THbis} for bistable ones:

\begin{theorem}\label{THmono}
 Let $f$ be monostable.
 
 Let $\Omega\subset\mathbb{R}^N$ be a $C^2-$regular domain. The system \eqref{maineq} can be stabilized:
 \begin{enumerate}
  \item  to $w\equiv 1$ for any admissible initial condition and any $\mu>0$.
  \item  to $w\equiv 0$:
  \begin{itemize}
  \item   for any initial condition iff $\mu^*(\Omega,f)<\mu$,
  
   \item for  $u_0$ with $\|u_0\|_{L^\infty}$ small enough if $$\frac{f'(0)}{\lambda_1(\Omega)}<\mu\leq\mu^*(\Omega,f),$$
  \end{itemize}
 \end{enumerate}
 where $\mu^*(\Omega,f)$ defined in Definition \ref{defpl} is a positive constant and $\lambda_1(\Omega)$ is the first eigenvalue of the Dirichlet Laplacian
\end{theorem}

\begin{remark}
 In the monostable case, there might exist a barrier even if the solution $w\equiv 0$ is still locally stable. In this situation, we might still approach $w\equiv 0$ for certain initial data. However, when $w\equiv 0$ becomes unstable it is not possible to approach $w\equiv 0$ without violating the state constraints. 
\end{remark}
\begin{remark}
 For the case of the Fisher-KPP nonlinearity $f(s)=s(1-s)$ the situation in which a barrier exists together with the fact that $w\equiv 0$ is stable does not occur \cite{PLLEPSSEQ,BERESTYCKI-FRENCH}. 
\end{remark}
\begin{remark}
 Physically it makes sense also to remove the  upper constraint in the state, i.e. to ask only for $u\geq 0$. In this situation, one can see that $w\equiv 1$ can be reached in finite time but not arbitrarily small \cite{DARIO}.
\end{remark}

    Now we turn our attention to bistable nonlinearities. The main controllability result we prove is the one concerning the steady-state $w\equiv\theta$ as a target. The most straightforward strategy would be; first, to set the constant control $a \equiv \theta$ during a long time interval, aiming to approach the target $w \equiv \theta$, and later to complement with a local controllability result. Nevertheless, this strategy does not suffice and does not lead to an optimal result when $0<\mu\leq \mu^*_\theta(\Omega,f)$.
    Then, to find an admissible control that can steer the solution to $w\equiv\theta$ is a priory more challenging. As we shall see, all initial data in the basin of attraction of $w\equiv 0$ , $\mathcal{A}$, can be controlled to the steady-steady  $w\equiv\theta$, independently of the stability properties of the target steady-state and the multiplicity of steady-states with the same boundary conditions.
    
    Note however, the existence of barrier functions already tells us that $\mathcal{A}$ necessarily does not contain the whole set of initial data $0\leq u_0 \leq 1$ for certain ranges of $\mu>0$. The requirement $u_0\in\mathcal{A}$ for controlling to $w\equiv \theta$ is a necessary and sufficient condition: whenever we are not in $\mathcal{A}$ we do not have controllability to $w\equiv\theta$.

 \begin{theorem}\label{THbis}
  Let $f$ be bistable with $F(1)>0$. Let $\Omega\subset\mathbb{R}^N$ be a $C^2$-regular domain.
  \begin{enumerate}
  
                                                                                                           \item\label{th1}  For any $\mu>0$ and any initial data $u_0\in L^\infty(\Omega;[0,1])$, the solution can be stabilized to $w\equiv 1$. More precisely, the solution taking boundary values $a(x,t)=1$ tends to $w\equiv 1$ as $t$ tends to infinity.
\item\label{th2}  $\mathcal{A}=L^\infty(\Omega;[0,1])$ iff $\mu>\mu^*(\Omega,f)$, where $\mathcal{A}$ is the basin of attraction of $w\equiv 0$ (Definition \ref{basin}),  where $\mu^*(\Omega,f)$ is the positive constant in Definition \ref{defpl}. 
\item\label{th3} For $\mu^*(\Omega,f)\geq \mu>0$, if $u_0\notin\mathcal{A}$, then $u_0$ is not controllable to $w\equiv\theta$. 
                      
   \item\label{th4} If $\mu>\mu^*(\Omega,f)$, there exists $T_{\mu}^*>0$ (uniform with respect to $u_0$ but not with respect to $\mu$) such that for every $T\geq T_{\mu}^*$ system \eqref{maineq} is controllable to $w\equiv\theta$ for any initial data $u_0\in L^\infty(\Omega;[0,1])$ by means of a function $a\in L^\infty(\partial\Omega\times[0,T],[0,1])$.
   \item\label{th5} If $\mu^*(\Omega,f)\geq\mu>0$ and $u_0\in\mathcal{A}$, there exists $T^*_{u_0,\mu}>0$ such that for every $T\geq T^*_{u_0,\mu}$ system \eqref{maineq} is controllable to $w\equiv\theta$ by means of a function $a\in L^\infty(\partial\Omega\times[0,T],[0,1])$. Moreover, $T^*_{u_0,\mu}$ is not uniformly bounded for all $u_0\in\mathcal{A}$.
  \end{enumerate}

 \end{theorem}
 The proof of the three first points will be an immediate consequence of the propositions in Section \ref{SMultiplicity}.
 
For $\mu^*(\Omega,f)\geq \mu>0$, we will see that minimal barriers with respect to the $L^\infty$-norm are at the border of $\mathcal{A}$. The nonuniform controllability time of point \ref{th5} of Theorem \ref{THbis} comes by a contradiction argument, taking a sequence of initial data in $\mathcal{A}$ tending to a minimal barrier with respect to the $L^\infty$-norm. 

 \begin{remark}
 Proposition \ref{upboundmu} and Proposition \ref{Bistable} give upper and lower bounds for $\mu^*(\Omega,f)$. For instance, for the bistable one-dimensional case we have that: 
 $$\frac{F(1)^2}{8(F(1)-F(\theta))}\leq\mu^*([0,1],f)\leq \frac{\max_{s\in[0,1]}\frac{f(s)}{s}}{\pi^2}.$$
\end{remark}

   Finally, we state the result for the nonlinearities that satisfy $F(1)=0$. In this case, the traveling wave solutions for the Cauchy problem are stationary. For the prototypical nonlinearity $f(s)=s(1-s)(s-\theta)$, it corresponds to the case in which $\theta=1/2$. 
  
  \begin{theorem}\label{THbis12}
  Let $f$ be bistable with $F(1)=0$ and $\Omega\subset \mathbb{R}^N$ be a bounded $C^2$-regular domain.  For any $\mu>0$ we have that:
  \begin{enumerate}
   \item   $w\equiv 0$ and $w\equiv 1$ are the unique admissible steady solutions for boundary control $a=0$ and $a=1$, respectively. Thus,  for all $\mu>0$ and  initial data $u_0\in L^\infty(\Omega;[0,1])$ the solution reaches asymptotically $w\equiv0$ (resp $w\equiv1$) with boundary control $a=0$ (resp $a=1$).
   \item For every $\mu>0$, there exist a time $T_\mu^*$ (uniform with respect to $u_0\in L^\infty(\Omega;[0,1])$ but not with respect to $\mu$) such that for all $T\geq T_\mu^*$ there exists a function $a\in L^\infty(\partial\Omega\times[0,T],[0,1])$ such that the solution of system \eqref{maineq} at $T$ is equal to $w\equiv \theta$.
  \end{enumerate}
 
 \end{theorem}

 Several remarks are in order:
 \begin{remark}
 It is known that the fact of having constraints implies that there exists a minimal controllability time \cite{DARIO} (see also \cite{LOHEAC}).
\end{remark}

 \begin{remark}
  Notice that $\mathcal{A}$ contains $L^\infty(\Omega;[0,\theta])$ for all $\mu>0$.  Moreover, in Section \ref{Aexample}, using the phase-plane analysis and the comparison principle, we will give more examples about the initial data lying in $\mathcal{A}$.
 \end{remark}

 \begin{remark}
  
  In fact, the path of steady-states we shall build for all $\mu>0$  in Claim \ref{constr} (see its representation in Figure \ref{Illustration}), linking $w\equiv 0$ and $w \equiv \theta$, passes by a continuum of intermediate steady-states. And the arguments we employ to show the control of the system in finite time to the final target $w \equiv \theta$ could also be applied to control the system to any of the other intermediate steady-states.
  \end{remark}

  \begin{remark}\label{critRemark}
    In this remark, we will see two important limitations of the staircase method.
    \begin{enumerate}
     \item Not every pair of admissible steady-states are connected by an admissible continuous path. 
     \item The staircase method or, more precisely, the existence of an admissible path of steady-states, is a sufficient but not a necessary condition to ensure controllability. Even if there does not exist an admissible path between two steady-states, we might have controllability.
    \end{enumerate}
  Let us give some more details.
    \begin{enumerate}
     \item We say that two $\omega$-limit sets are ordered, and we denote it by $S_1<S_2$, if for every $v\in S_1$ and every $w\in S_2$ one has that $v<w$. In Theorem \ref{THbis}, we also notice a fundamental limitation of the staircase method:
  \begin{center}
         \textit{We cannot aim to construct admissible paths between two steady-states that have ordered $\omega$-limits for controls $a=0$.}´                                                                                                                                                                                                                                                   
                                                                                                                                                                                                                                                                    \end{center}
                                                                                                                                                                                                                                                                   
      If there exist an admissible path between two steady-states, say $u_1$ and $u_2$, this means that we can control from $u_1$ to $u_2$ and vice versa, since we can take the path oppositely. 
      Indeed, for every $T>0$ and for all $a\in L^\infty(\partial \Omega\times [0,T];[0,1])$, denote by $u(t;a,u_0)$ the associated state with control $a$ and initial datum $u_0\in L^\infty(\Omega;[0,1])$, by the maximum principle one has that:
      \begin{equation}\label{omegamax}
       u(t;a,u_2)\geq u(t;a=0,u_2)\quad \forall t\in[0,+\infty]
      \end{equation}
      
      If an admissible path exists between $u_1$ and $u_2$, we would have controllability, namely, there would exist  $T_1\in\mathbb{R}^+$ and $a_1\in L^\infty(\partial\Omega\times[0,T_1];[0,1])$ such that $u(T_1;a_1,u_2)=u_1$.
      
      Then setting
      \begin{equation*}
       a^*:=\begin{cases}
             a_1\quad& t\in[0,T_1]\\
             0\quad &t\in(T_1,+\infty)
            \end{cases}
      \end{equation*}
      we would have a trajectory that would violate \eqref{omegamax} since
      \begin{equation*}
       \omega_{a^*}(u_2)= \omega_0(u_1)<\omega_0(u_2).
      \end{equation*}

     Note that, for applying the contradiction argument with the maximum principle, more relaxed requirements on the $\omega$-limits could be taken.
     
     The same situation would hold for $a=1$ in case that there were admissible nontrivial steady solutions with boundary $a=1$, like in the case of heterogeneous drifts \cite{drift}.
          
     It is important to notice that this limitation is due to pointwise constraints on the controlled states and on the maximum principle. In the absence of such constraints, one can admit more elements in the path, as in  \cite{CORON-TRELAT}, and this nonexistence argument for the path of steady states does not apply.

     \item The staircase method is not the only possible way to control. Given two steady-states $u_1$ and $u_2$ with ordered $\omega$-limits for $a=0$, in which case, it does not exist a path between $u_1$ and $u_2$, it does not mean that we cannot control from one to the other. 
     \begin{center}
      \textit{When constraints are present, we might be able to control from $u_1$ to $u_2$ but not from $u_2$ to $u_1$.}
     \end{center}
  
  Theorem \ref{THbis}, in particular, points that for every $\mu>0$ we are always able to approach the steady-state $w\equiv 1$ taking controls $a=1$ from any initial data in $L^\infty(\Omega,[0,1])$. However, for $0<\mu\leq \mu^*(\Omega,f)$ a barrier exists and blocks the stabilization to $w\equiv0$ for any initial data above a barrier.  
  
  One can also find examples in which the controllability from $u_1$ to $u_2$ is in a finite time horizon holds, but the reverse is impossible, i.e., one might not be able to control $u_2$ to  $u_1$ neither in finite or infinite time horizon. For $0<\mu\leq \mu^*(\Omega,f)$, we can find, for instance, an admissible target steady-state $u_2$ with Dirichlet trace away of $0$ and $1$ that is in the basin of attraction of the maximal barrier with respect to the $L^\infty$-norm that is controllable for any initial steady-state $u_1\in\mathcal{A}$ and $T\geq T^*_{u_0,\mu}>0$.  
However, this is an irreversible control process, from $u_2$ as initial steady-state we cannot control to any $u_1\in\mathcal{A}$.
    \end{enumerate}

 \end{remark}
 
 \begin{remark}
  In Section \ref{Sproofs}, we will construct a path that connects $w\equiv 0$ with $w\equiv \theta$. For the argument made in Remark \ref{critRemark}, if we consider all the elements of the path being initial data for the control $a=0$, the $\omega$-limit for all the initial data considered will be $\{w\equiv 0\}$. Which corresponds to say that all the elements in the path constructed are in $\mathcal{A}$.
 \end{remark}

\begin{remark}

In both cases $F(1) > 0$ or $F(1) = 0$, when the domain $\Omega$   is a ball and the initial data is radially symmetric, the Dirichlet control can be chosen to be radially symmetric as well.
\end{remark}
\begin{remark}

  Furthermore, when $\Omega$ is a ball and $0<\mu\leq\mu^*(\Omega,f)$,  the arguments in the construction of the path can be extended to a path that connects $w\equiv 0$ with the minimal barrier with respect to the $L^\infty$-norm (Section \ref{pathexample}). This means that we can build a control that stabilizes, in an open-loop manner, around the minimal barrier, provided that $u_0\in\mathcal{A}$. 
 
\end{remark}

The structure of the paper is the following:
\begin{itemize}
 \item First, in Section \ref{PREL}, we recall the main results on asymptotic dynamics of the semilinear heat equation \cite{lions1984structure,CAZENAVE-HARAUX}, the staircase method \cite{DARIO}, and the approach taken in dimension one, \cite{POUCHOL}.
 \item Afterwards, in Section \ref{Tlema}, we proceed to provide some technical lemma that are needed for the proofs of the theorems above.
 \item Then, in Section \ref{Sproofs}, we give the proof of the Theorems \ref{THmono}, \ref{THbis} \ref{THbis12}, provide an example of a path arriving to the minimal barrier with respect to the $L^\infty$-norm and more insights about $\mathcal{A}$.
 \item In Section \ref{Snum}, we provide numerical simulations of the construction of the path, an implementation of the quasistatic control in Ipopt, and we give a numerical example of the control and controlled state in minimal time.
 \item Finally, in Section \ref{Sconcl}, we set the conclusions and discuss future perspectives on the topic.

\end{itemize}

\section{Preliminary results}\label{PREL}
\subsection{Asymptotic behavior}
The semilinear heat equation
\begin{equation}\label{SHE}
 \begin{dcases}
  u_t-\mu\Delta u=f(u)&\quad (x,t)\in\Omega\times(0,T],\\
  u=0&\quad (x,t)\in\partial\Omega\times(0,T],\\
   u(0)=u_0\in L^\infty(\Omega) &\quad x\in\Omega,
 \end{dcases}
\end{equation}
where $f$ is globally Lipschitz, is as a gradient dynamical system on the metric space $C_0(\Omega)$,
\begin{equation*}
 u_t=-\nabla_uJ[u],
\end{equation*}
where $J[u]=\int_\Omega \frac{\mu}{2} |\nabla u|^2-F(u)dx$.
$J$ acts as a strict Lyapunov functional.
Due to this fact, the solution of the semilinear heat equation, whenever its trajectory is globally defined and bounded, it approaches the set of steady-states:
\begin{theorem}[Theorem 9.4.2 in \cite{CAZENAVE-HARAUX}]
Let $\mathcal{S}_E:=\{v\in\mathcal{S}\quad\text{such that } J(v)=E\}$ and $J$ be coercive. Then the solution $u$ of \eqref{SHE} and $\mathcal{S}$ satisfy:
\begin{itemize}
 \item  $J(u(t))\to E$,
\item   $\mathcal{S}_E\neq \emptyset$,
\item $d(u(t),\mathcal{S}_E)\to 0$ as $t\to\infty$, where $d$ is the distance in $C_0(\Omega)\cap H_0^1(\Omega)$.
\end{itemize}
\end{theorem}

By the LaSalle invariance principle, if the set of steady-states consists of discrete points, the solution converges to one of them. Generally, initial data converge to a particular steady-state, see \cite{lions1984structure}. However, there are tricky examples, involving space dependence of $f$, $f(x,u)$, in which there can exist bounded trajectories that are nonconvergent for $N\geq 2$ (see \cite{POLACIK1996,POLACIK2002}). In the case that $f$ is analytic, the convergence to steady-states is guaranteed \cite{SIMON83,JENDOUBI1998,HARAUXJENDOUBI}. This is a feature of $N\geq 2$, since, in the one-dimensional case, the convergence of bounded trajectories is guaranteed for any $C^2$ nonlinearity by Matano's result \cite{MATANO1978}.

\subsection{Staircase method}
  The primary tool for understanding how we can reach the steady-state $w\equiv \theta$ will be the staircase method.

The following Theorem in \cite{DARIO} ensures that if we find an admissible continuous path of steady-states between two steady-states, then we are able to find a control function that steers one steady-state to the other.
Consider
\begin{equation}\label{Dpb}
 \begin{dcases}
  v_t-\mu\Delta v =f(v)\quad & (x,t)\in \Omega\times (0,T],\\
  v=a(x,t)\quad & (x,t)\in\partial \Omega\times (0,T],\\
  0\leq v(0,x)=v_0(x)\leq 1 & x\in\Omega,
 \end{dcases}
\end{equation}
where $\Omega$ is a bounded domain with boundary $C^2$. We say that $v_1$ and $v_0$ are path connected steady-states if there exist a continuous function (with respect to $\|\cdot\|_{L^\infty}$)  from $[0,1]$ to the set of steady-states $\gamma:[0,1]\to S$ such that $\gamma(0)=v_0$ and $\gamma(1)=v_1$. Denote by $\overline{v}^s:=\gamma(s)$.
\begin{theorem}[Theorem 1.2 in \cite{DARIO}]\label{DTHM}
 Let be $v_0$ and $v_1$ be path connected bounded steady-states. Assume there exists $\nu>0$ such that:
 \begin{equation*}\label{bordercondition}
  \nu\leq\overline{v}^s\leq 1-\nu\quad\text{ a.e. on } \partial\Omega
 \end{equation*}
 for any $s\in[0,1]$. Then, if $T$ is large enough, there exist a control function $a\in L^\infty\left(\partial\Omega\times (0,T) ; [0,1]\right)$ such that the problem \eqref{Dpb} with initial datum $v_0$ and control $a$ admits unique a solution verifying $v(T,\cdot)=v_1$ and $0\leq a\leq 1$ on $(0,T)\times \partial \Omega$.

\end{theorem}

The proof is based on extracting a finite sequence of points of the continuous path of steady-states and apply local controllability between them for going from one to another until the target is reached. This is the reason why it is called the staircase.

\subsection{One dimensional approach}
For proving the controllability to $w\equiv \theta$ under the prescribed state constraints, one strategy is to find an admissible path of steady-states \cite{POUCHOL}. The construction of a path of steady-states for controlling the one-dimensional semilinear heat equation was firstly considered in \cite{CORON-TRELAT}. For finding the admissible path, one sees the steady-states for the problem:
\begin{equation*}
\begin{cases}
  u_t-u_{xx}=f(u),\\
 u(0,t)=a_1(t),\quad u(L,t)=a_2(t),\\
 0\leq u(x,0)\leq 1,
\end{cases}
\end{equation*}
as an ODE system:
\begin{equation}\label{1dstatic}
 \frac{d}{dx}\begin{pmatrix}
              v\\
              v_x
             \end{pmatrix}=\begin{pmatrix}
             v_x\\
             -f(v)
             \end{pmatrix},
\end{equation}
for initial conditions:
\begin{equation*}
 \begin{pmatrix}
  v^{(s)}(0)\\
    v_x^{(s)}(0)
 \end{pmatrix}
=\begin{pmatrix}
  s\alpha+(1-s)\theta\\
  s\beta
 \end{pmatrix}.
\end{equation*}
The key point is to find an admissible invariant region 
for the dynamics \ref{1dstatic}
that ensures that all steady-states will be admissible. Remind that an invariant region in the phase plane is a set $\Gamma\subset\mathbb{R}^2$
such that for every $(v^0,v^0_x)\in \Gamma$ the solution of \eqref{1dstatic} 
with initial data $(v^0,v^0_x)$ remains inside $\Gamma$ for all $x\geq 0$.
Moreover, if the invariant region $\Gamma$ is such that $\forall (v,v_x)\in \Gamma$
we have that $v\in[0,1]$ then it will be an admissible invariant region.

Then by solving the ODE, we will obtain the required boundary conditions. Moreover, by continuous dependence on the initial data, we have that the path is continuous. Figure \ref{fig1dstatic} illustrates the procedure to build the path of steady-states and the invariant region for \eqref{1dstatic}.

\begin{figure}
    \begin{center}

\includegraphics[scale=0.5]{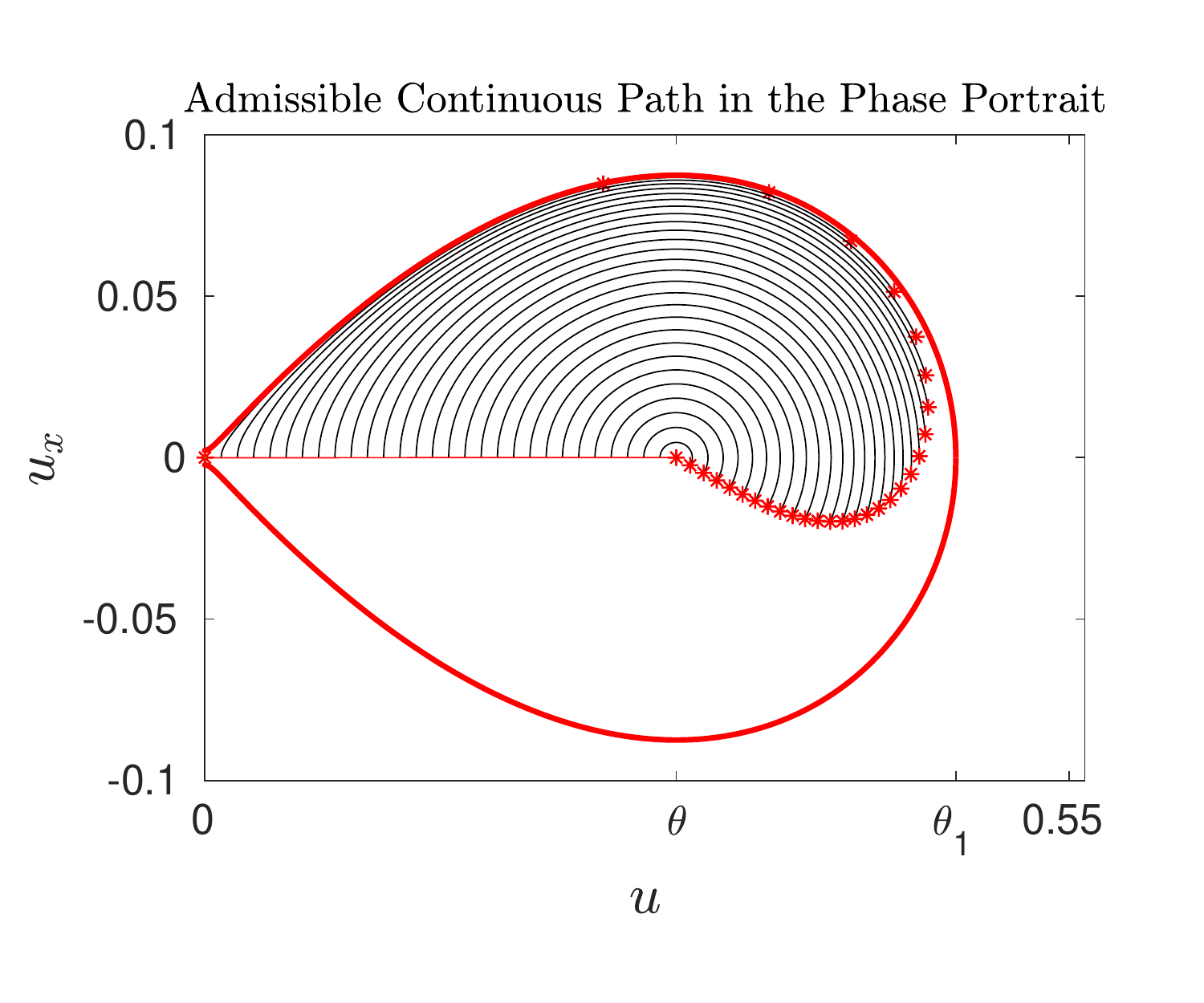}
\includegraphics[scale=0.5]{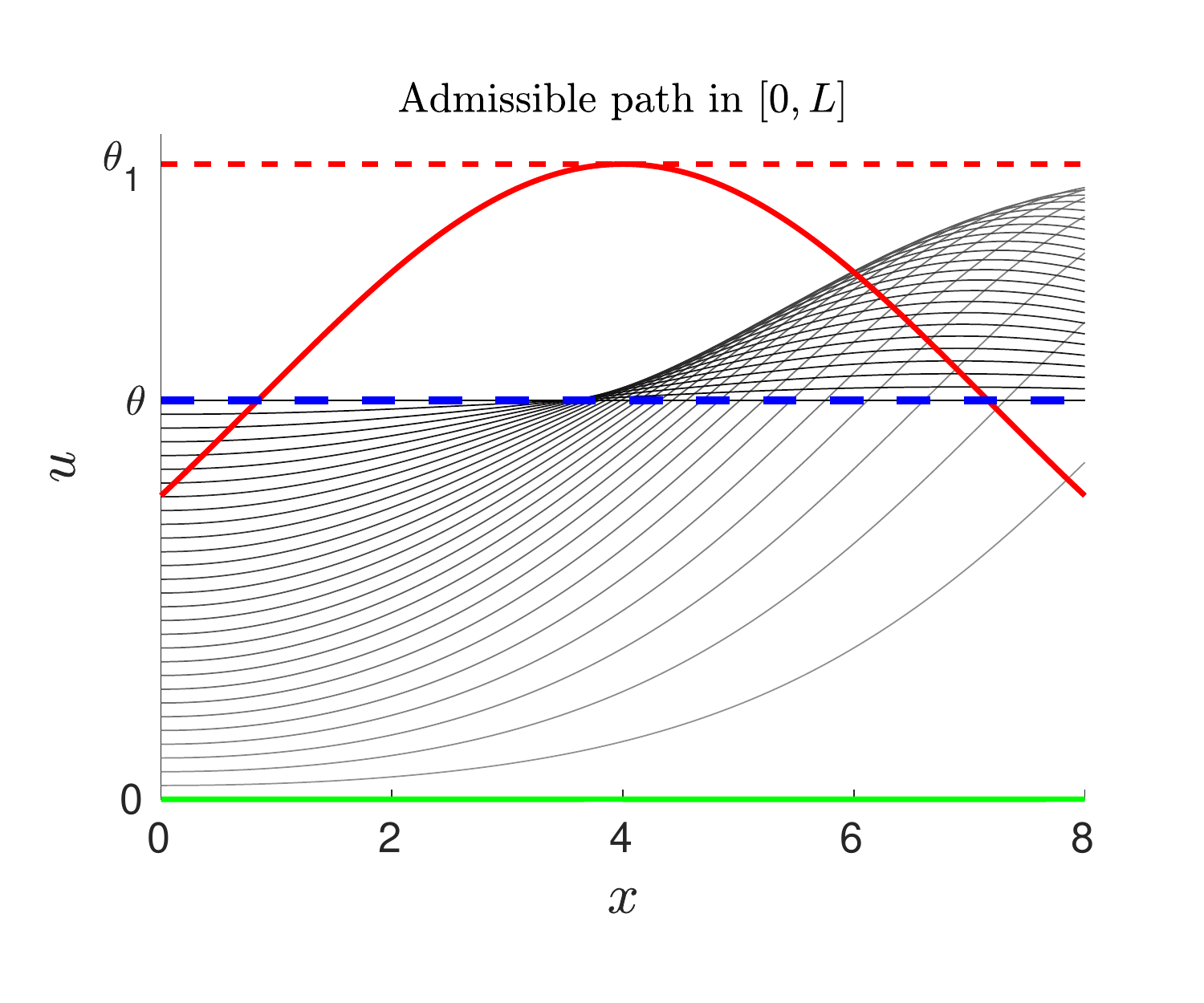}
 \caption{(Left) Phase portrait and generation of the path of steady-states via solving \eqref{1dstatic} from $0$ to $L$.
 (Right) The corresponding continuous path in its natural domain $[0,L]$, (red) the solution delimiting the invariant region.}\label{fig1dstatic}
 \end{center}
\end{figure}
\begin{remark}
 In the one dimensional approach the authors \cite{POUCHOL} work in the framework of variable $L$ instead of considering the diffusivity $\mu$.
\end{remark}

\section{Technical Lemmas}\label{Tlema}

\subsection{Estimates on the existence of nontrivial solutions}\label{SMultiplicity}
In this subsection, we study the existence of nontrivial solutions of the boundary value problems around our steady-states of interest.

\begin{prop}[An upper bound for $\mu^*$]\label{upboundmu}
 Let $f:\mathbb{R}\to \mathbb{R}$, assume that $f$ is bounded. Assume furthermore that $f(0)=0$
 Consider:
 \begin{equation}\label{fineq}
 \begin{dcases}
    -\mu\Delta v= f(v)&\quad x\in\Omega,\\
  v=0&\quad x\in\partial\Omega,\\
  1>v>0&\quad x\in\Omega.
 \end{dcases}
 \end{equation}
Then,
\begin{equation*}
 \mu^*\leq \frac{\max_{s\in[0,1]}\frac{f(s)}{s}}{\lambda_1(\Omega)},
\end{equation*}
where $\lambda_1(\Omega)$ is the first eigenvalue of the Dirichlet Laplacian.
\end{prop}
The proof is classical and follows by comparison arguments and energy estimates \cite{PLLEPSSEQ,HERVE}.

\begin{prop}[A lower bound for $\mu^*$]\label{Bistable}
 
  Assume that $f(0)=f(\theta)=f(1)=0$, and that $f'(0)<0$, $f'(1)<1$, $f'(\theta)>0$. Moreover consider $F(v)=\int_0^v f(s)ds$ and assume that $F(1)>0$.
  Consider $\Omega\subset\mathbb{R}^N$ with boundary $C^2$ 
  and problem \eqref{fineq}.
  Denote by $B_\Omega$ a ball of maximal measure inside $\Omega$, $B_\Omega\subset\Omega$.
  Then, for any $\mu>0$ fulfilling 
   \begin{equation*}
     \mu<\frac{2\delta^2\Gamma\left(\frac{N}{2}+1\right)^{2/N}\left(F(\theta)+(1-\delta)^N(F(1)-F(\theta))\right)m(B_\Omega)^{2/N}}{\pi \left(1-(1-\delta)^N\right)},
    \end{equation*}
  there exists a solution of the problem \eqref{upbound}, where $\delta>0$ fulfills
    \begin{equation*}
   \delta < 1-\left(\frac{-F(\theta)}{F(1)-F(\theta)}\right)^{1/N},
  \end{equation*}
  and $\Gamma$ is the gamma function.
  This provides a lower bound of $\mu^*$.
 \end{prop}

 \begin{remark}
  One can obtain a similar threshold for the monostable nonlinearity following the same argument of the proof.
 \end{remark}
 \begin{remark}
  The proof is based on finding a test function $v_\delta$ and on seeing for which choice of $\mu>0$ and $\delta>0$, we can guarantee that $J[v_\delta]<0$. Since $J[w\equiv 0]=0$ for all $\mu>0$ we have that for the previous choice of $\mu$, there will exist a nontrivial solution.
 \end{remark}
 We postpone the computation for the Appendix.
 
	For the bistable case, a barrier exists, has its maximum value above $\theta$
\begin{prop}[Maximum of positive solutions]\label{MaxPosSol}
            Let $u$ be a solution of \eqref{fineq} with $f$ being bistable then the maximum of $u$ in $\Omega$ is above $\theta$:
            \begin{equation*}
                \max_{x\in\Omega}u(x)>\theta.
            \end{equation*}

        \end{prop}
        \begin{proof}
         The proof follows by contradiction. Assume that the maximum of $u$ is lower or equal than $\theta$, then the energy estimate gives us the contradiction:
         \begin{equation*}
            0<\int_{\Omega} \mu| \nabla u|^2 dx=\int_{\Omega} uf(u)<0,
         \end{equation*}
         where the strict inequality in the left hand side comes from the assumption that the solution is nontrivial and the right hand side inequality comes from the fact that $f$ is negative in $(0,\theta)$.

        \end{proof}
        
        \begin{remark}
         Note that the fact that the maximum value of a nontrivial solution with boundary value $0$ is above $\theta$ implies that we cannot reach the steady-state $w\equiv 0$ asymptotically neither the steady-state $w\equiv \theta$ if we start with an initial value above this nontrivial solution.
        \end{remark}

 In the following proposition we show that there could exist a regime of $\mu$, $(\mu^*(\Omega,f),\mu^*_\theta(\Omega,f)]$ for which there is no barrier but the trivial control strategy of setting $a\equiv\theta$ for a long time plus local control might not work.
\begin{prop}[Order in the thresholds]
For $f$ bistable, when $F(1)>0$ we have that $\mu_\theta^*\geq \mu^*$.

\end{prop}
\begin{proof}
 The result follows from the elliptic comparison principle \cite{HERVE} together with the fact that any nontrivial solution of the boundary value problem has its maximum above $\theta$.
\end{proof}

\begin{remark}[Bounds on $\mu_\theta^*$]\label{boundstheta}
 When we study nontrivial solutions for the problem \eqref{maineq} with Dirichlet boundary conditions equals to $\theta$, we can reduce our analysis to the study of a monostable nonlinearity, since after the change of variables $u=v-\theta$, the nonlinearity ends up being $f(v+\theta)$, which is monostable. The first nontrivial solution that appears is not going to change sign, because otherwise, such oscillating solution $w$ will have a positive (or negative) part which will be a subsolution (or a supersolution) of the problem by extending it with $0$ (in the new coordinates) outside the positive (or negative) part, leading to the existence of a nontrivial solution that does not change sign.
\end{remark}

\begin{remark}
 Note that the stability of the stationary solution $w\equiv \theta$ becomes more and more unstable as the diffusivity decreases. However, the number of eigenfunctions that are unstable is only finite.
\end{remark}

Now we turn our attention to the existence or nonexistence of
elliptic solutions with Dirichlet boundary $1$. The presence of traveling waves for the one-dimensional case has been widely studied \cite{KOLMOGOROV37,FIFEBOOK,PERTHAME}. One can easily see that the same functions constantly extended in the $N-1$ remaining space dimensions are also solutions of the Cauchy problem in $\mathbb{R}^N$. 

\begin{prop}[Traveling waves and convergence to  $w\equiv 1$ for any domain]\label{TWto1}
The existence of a decreasing traveling wave implies that for any initial admissible condition and every domain, the solution can be asymptotically driven towards $w\equiv 1$.

If $f$ is monostable or $f$ is bistable with $F(1)>0$ then there is a unique solution of the boundary value problem
         \begin{equation*}\label{feina1}
            \begin{dcases}
                - \mu\Delta u =f(u)\quad&\text{in }\Omega,\\
                {0\leq u\leq 1}\quad &\text{in }\Omega,\\
                u=1\quad&\text{in }\partial\Omega
            \end{dcases}
    \end{equation*}
and for any domain $\Omega$, we have that any initial admissible condition can be asymptotically driven towards the steady-state $w\equiv 1$. 
\end{prop}
    
    \begin{proof}
     It is known that the problem:
     \begin{equation*}
       u_t-\mu\Delta u=f(u)\qquad (x,t)\in\mathbb{R}^N\times(0,+\infty),\\
     \end{equation*}
has a traveling wave solution for every $\mu$. Furthermore, the traveling wave profile takes values in $[0,1]$, and it is a monotone function decreasing in the direction of the velocity vector \cite{PERTHAME,FIFEBOOK}. The idea is to use a section of the traveling wave as a parabolic subsolution to our problem. 
Now we come back to our new (parabolic) problem;
     \begin{equation}\label{parfeina1}
            \begin{dcases}
                u_t-\mu\Delta u = f(u)\quad&(x,t)\in\Omega\times(0,T],\\
                u(x,t)=1\quad&(x,t)\in\partial\Omega\times(0,T],\\
                                0<u(x,0)<1 &x\in\Omega.
            \end{dcases}
    \end{equation}
     Since the traveling wave profile is monotone decreasing, we can consider a section of the traveling wave such that is below to $u(x,0)$ in $\Omega$, 
      let us denote by $TW(x)$ the maximum profile of traveling wave that satisfies:
     \begin{equation*}
      TW(x)\leq u(x,0)\quad\forall x\in \Omega. 
     \end{equation*}
     Now we note that the following problem:
     \begin{equation}\label{parfeina1comp}
            \begin{dcases}
                u_t- \mu\Delta u = f(u)\quad&(x,t)\in\Omega\times(0,T],\\
                u(x,t)=TW(x-c_{\mu}t)\quad&(x,t)\in\partial\Omega\times(0,T],\\
                                u(x,0)=TW(x) &x\in\Omega,
            \end{dcases}
    \end{equation}
    is a subsolution of \eqref{parfeina1} with $c_\mu>0$ for every $\mu>0$, then by the parabolic comparison principle we have that the solution of \eqref{parfeina1} will be above \eqref{parfeina1comp} and therefore the solution of \eqref{parfeina1} will converge to $w\equiv 1$.

    \end{proof}

\subsection{Radial solutions}
In this subsection, we discuss radial solutions of semilinear PDEs via ODE methods. The reason to do so is that the construction of the path towards $w\equiv \theta$ for the bistable case will rely on extending our domain $\Omega$ to a ball $\Omega\subset B_R$, construct the path for this ball and then restrict to our original domain $\Omega$.
Consider the following elliptic PDE:
\begin{equation}\label{pathmd}
\begin{dcases}
  -\mu\Delta u=  f(u)\qquad x\in B_r\subset \mathbb{R}^N,\\
  u(0)=a,\\
  Du(0)=0,
\end{dcases}
\end{equation}
where $f$ is globally Lipschitz. It is well known that the solutions of a semilinear elliptic equation in a ball are radially symmetric \cite{evans10}. We rewrite the \eqref{pathmd} in radial coordinates and absorbing $\mu$ in the nonlinearity the following is obtained:

\begin{equation}\label{radial}
\begin{dcases}
 u_{rr}(r)+\frac{N-1}{r}u_r(r)=-f(u(r))\qquad r\in [0,R_m),\\
u(0)=a,\\
  u_r(0)=0.
\end{dcases}
\end{equation}
The following Lemmas are devoted to the existence, uniqueness, continuous dependence with respect to parameters, and global definition of \eqref{radial} for certain ranges of $a$. The local existence and uniqueness follow from standard contraction argument \cite{CAZENAVE-SEMILINEARELL}.
\begin{lemma}[Local Existence and uniqueness]
 There exist a unique solution for $R_m$ small enough to \eqref{radial}
\end{lemma}
\begin{proof}
 First we proof the wellposedness of \eqref{radial}, since the term $\frac{1}{r}$ is not integrable. We proceed by multiplying by $r^{N-1}$ and integrating to obtain something of the form:
\begin{equation*}
 u(r)=a+\int_0^r \frac{1}{s^{N-1}} \int_0^s \sigma^{N-1} f(u(\sigma)) d\sigma ds
\end{equation*}
Now define the map:
\begin{equation*}
 Tu=a+\int_0^{R_m} \frac{1}{s^{N-1}} \int_0^s \sigma^{N-1} f(u(\sigma)) d\sigma ds
\end{equation*}
and we can show that it is a contraction for $R_m$ small enough:
\begin{align*}
        \|Tu-Tv\|_\infty&\leq \int_0^{r} \frac{1}{s^{N-1}} \int_0^s \sigma^{N-1} \|f(u(\sigma))-f(v(\sigma))\|_\infty d\sigma ds\\
        &\leq L\|u-v\|_\infty \int_0^{r} \frac{1}{s^{N-1}} \int_0^s  \sigma^{N-1}  d\sigma ds\\
        &\leq LR_m\|u-v\|_\infty
\end{align*}
where $L$ is the Lipschitz constant of $f$. Choosing $R_m$ small enough we have the contraction and hence the solution is unique.
\end{proof}

From now on we assume $f$ to be bistable with $F(1)\geq 0$. The following lemma is key for the admissibility of the path of steady-states.
\begin{lemma}[Invariant region]\label{Energy}
 Assume that $F(1)\geq 0$. Then there exists an admissible region in the phase space $\Gamma$ that is positively invariant.
\end{lemma}
\begin{proof}
Let us consider the following energy:
 \begin{equation*}
 E(u,v)=\frac{1}{2}v^2+F(u)
\end{equation*}
where $F(u)=\int_0^u f(s)ds$ and $v=u_r$. 
Define the following region:
\begin{equation*}
 D:=\{(u,v)\in\mathbb{R}^2 \quad \text{such that}\quad E(u,v)\leq 0\}
\end{equation*}
Let $\theta_1$ be defined as:
\begin{equation*}
\theta_1=\min_{s>0}\{F(s)=0\}
\end{equation*}

Note that the region defined by 
\begin{equation*}
 \Gamma:=\{(u,v)\in[0,\theta_1]\times\mathbb{R}\quad\text{such that}\quad |v|\leq \sqrt{-2F(u)}\}
\end{equation*}
Note that $\Gamma\subset D$.

Take $(u_0,0)\in \Gamma$, then the solution of \eqref{radial} with initial
datum $(u_0,0)$ satisfies:

\begin{align*}
 \frac{d}{dr}E(u,v)=vv_r+f(u)v=-\frac{N-1}{r}v^2<0
\end{align*}
So $(u,v)\in\Gamma$ for all $r>0$.
\end{proof}

Therefore radial solutions are globally defined and admissible for any $(u_0,0)\in\Gamma$.

\begin{remark}[$(\theta,0)\in \Gamma$]
 Note that $F(\theta)<0$ hence $\sqrt{-2F(u)}$ is well defined.
\end{remark}

\begin{remark}[Stationary Traveling waves and the invariant region]
 If $F(1)=0$ we obtain a region that is defined from $(0,0)$
 up to $(1,0)$ corresponding to the constant stationary solution
 $u(x)=1$.
 The traveling waves  in the one dimensional case, the one satisfying:
 \begin{equation*}
  \lim_{x\to -\infty}TW(x)=1,\quad \lim_{x\to +\infty}TW(x)=0
 \end{equation*}
 and the symmetric one satisfying 
 \begin{equation*}
  \lim_{x\to -\infty}TW(x)=0,\quad\lim_{x\to +\infty}TW(x)=1,
 \end{equation*}
 in the case that $F(1)=0$ are stationary,
 and define the aforementioned invariant region in the phase plane.
\end{remark}

\begin{lemma}[Continuous dependence]\label{Continuity}
 The solution of the initial value problem 
 \begin{equation*}
  \begin{dcases}
   \frac{d}{dr}\begin{pmatrix}
                u\\
                u_r
               \end{pmatrix}
=\begin{pmatrix}
                              u_r\\
                              -\frac{N-1}{r}u_r-f(u)
                             \end{pmatrix},\\
                             \begin{pmatrix}
                              u(0)\\
                              u_r(0)
                             \end{pmatrix}=\begin{pmatrix}
                             a\\
                             0
                             \end{pmatrix},
  \end{dcases}
 \end{equation*}
for $r\in[0,R_m)$ is continuous with respect to the initial condition $a$.
\end{lemma}
\begin{proof}
 Note that $\xi(r)=u(r)^2+u_r(r)^2$ satisfies the following differential inequality:
 \begin{align*}
  \frac{d}{dr}\xi(r)&=2uu_r+2u_r\left(\frac{N-1}{r}u_r-f(u)\right)\\
  &\leq 2uu_r(1+L)-2\frac{N-1}{r}u_r^2\\
  &\leq (1+L)\xi(r)
 \end{align*}
 Applying Gromwall's inequality the result follows.

\end{proof}

\section{Proofs of Theorems and examples}\label{Sproofs}

    \subsection{Convergence to $w\equiv 1$ for monostable and bistable nonlinearities}
    The existence of a traveling wave for all $\mu>0$ with $c_\mu>0$ implies the convergence to $w\equiv 1$ for $a(x,t)=1$ for every $\mu>0$, Proposition \ref{TWto1}. The argument applies for the monostable case and the bistable case with $F(1)>0$.

    \subsection{Convergence to $w\equiv 0$ for monostable and bistable nonlinearities with $F(1)>0$.}
       Discussion on the existence of a nontrivial steady-states already done previously in Proposition \ref{upboundmu} implies that we cannot converge to $w\equiv 0$ for every $\mu>0$.
        For $\mu>\mu^*(\Omega,f)$, $w\equiv 0$ is the only stationary solution that takes values between $0$ and $1$. Therefore, in this case, both for the monostable or bistable case, $\mathcal{A}=L^\infty(\Omega,[0,1])$.
        
        For the monostable case, when $$\mu^*(\Omega,f)\geq \mu>\frac{f'(0)}{\lambda_1(\Omega)},$$ $\mathcal{A}\neq\{w\equiv 0\}$ and for initial data small enough one can still stabilize around $w\equiv 0$. However, for $\mu<f'(0)/\lambda_1(\Omega)$, $w\equiv 0$ is unstable and $\mathcal{A}=\{w\equiv 0\}$. In the critical case $\mu=f'(0)/\lambda_1(\Omega)$, the stability depends on the nonlinear terms arising from higher order Taylor expansions around $0$ of $f$.
        
        For the bistable case, when $\mu^*(\Omega,f)\geq \mu$, $\mathcal{A}\neq L^\infty(\Omega,[0,1])$. However, in contrast with the monostable case,  for any $\mu>0$, $\mathcal{A}\neq\{w\equiv 0\}$ since $f'(0)<0$.

    \subsection{Proof of the controllability to $w\equiv\theta$ in Theorem 1.2  ($F(1)>0$ case).}
    Here we prove points \ref{th4} and \ref{th5} of Theorem 1.2 (point \ref{th3} is a direct consequence of Proposition \ref{MaxPosSol}).
    
    The control strategy consists of two phases. First, to approach $w\equiv 0$ in a long time interval by keeping the control $a=0$, and then, by following the path of steady-states, to reach the desired target.
  
    First of all, in Claim \ref{constr}, we build for all $\mu>0$ a continuous path of admissible steady-states that connects $w\equiv0$ with $w\equiv \theta$.

                \begin{claim}[Construction of the path]\label{constr}
                For every $\mu>0$ and every $\Omega$ there exists an admissible continuous path of steady-states connecting $w\equiv 0$ and $w\equiv \theta$.
            \end{claim}
            \begin{proof}
                1. Since our domain $\Omega$ is bounded we can find a ball with a big enough radius such that $\Omega\subset B_R$ Figure \ref{Eball}.
                \begin{figure}
                    \begin{center}

                \includegraphics[scale=0.3]{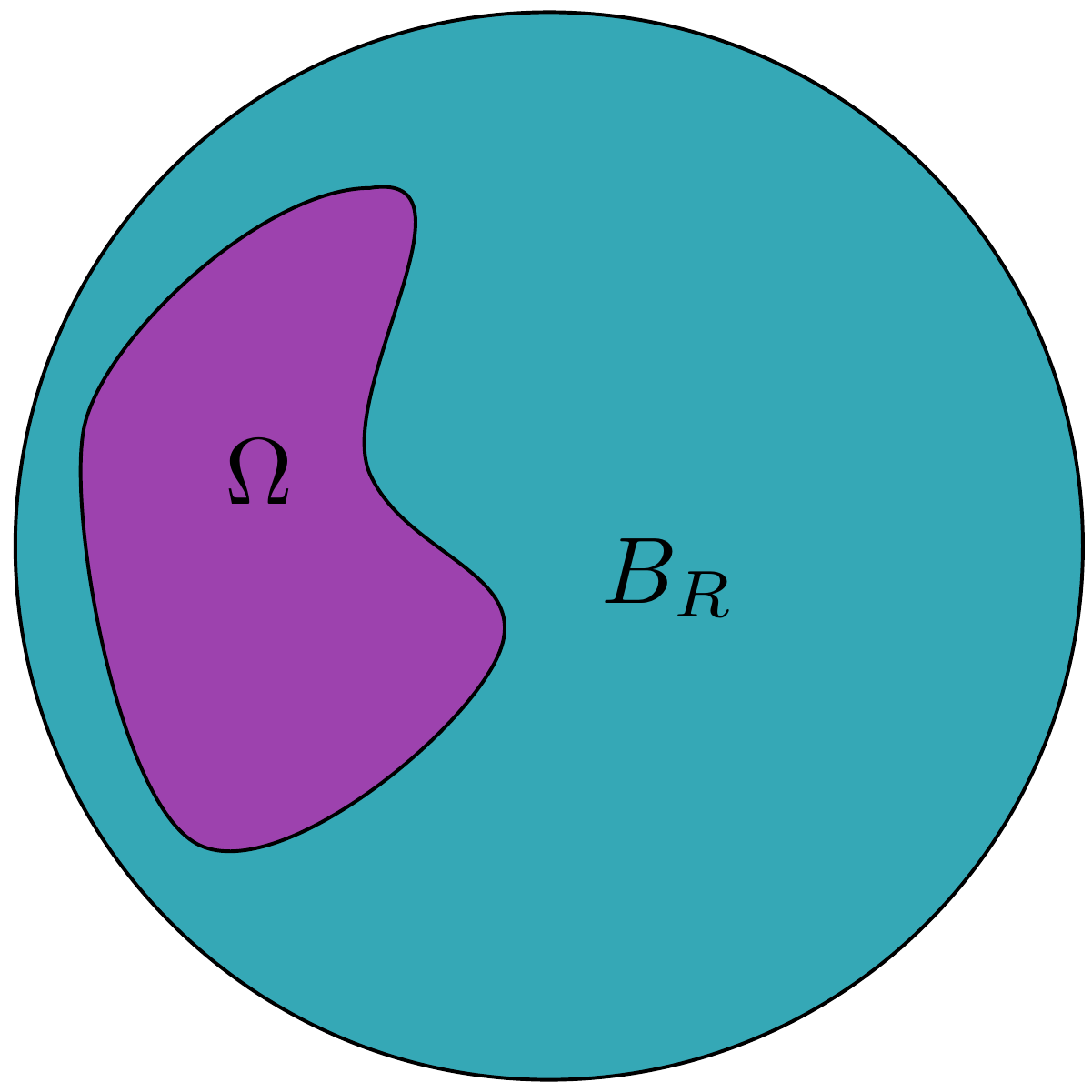}
                \caption{Domain $\Omega$ and a ball containing it where the path
                of steady-states will be constructed.}\label{Eball}
                \end{center}
                \end{figure}

                2. Construction of the path of steady-states on $B_R$:

                Let $R>0$ be an arbitrary positive number. Consider the one parameter family of solutions of the following Cauchy problem where $\mu$ has been absorbed by the nonlinearity:
                \begin{equation*}
                    \begin{dcases}
                    u^{(a)}_{rr}(r)+\frac{N-1}{r}u^{(a)}_r(r)=-f(u^{(a)}(r))\qquad r\in [0,R],\\
                    u^{(a)}(0)=a\in[0,\theta],\\
                    u^{(a)}_r(0)=0.  
                    \end{dcases}
                \end{equation*}
                Applying Lemma \ref{Energy} we obtain that the solutions are globally defined and applying Lemma \ref{Continuity} we obtain the continuous dependence with respect to initial data which implies the continuity of the path.

                Figure \ref{FigPhRad} shows the admissible invariant region and the connected path of steady-states connecting $w\equiv 0$ to $w\equiv \theta$.

                \begin{figure}
                    \begin{center}
                    \includegraphics[scale=0.5]{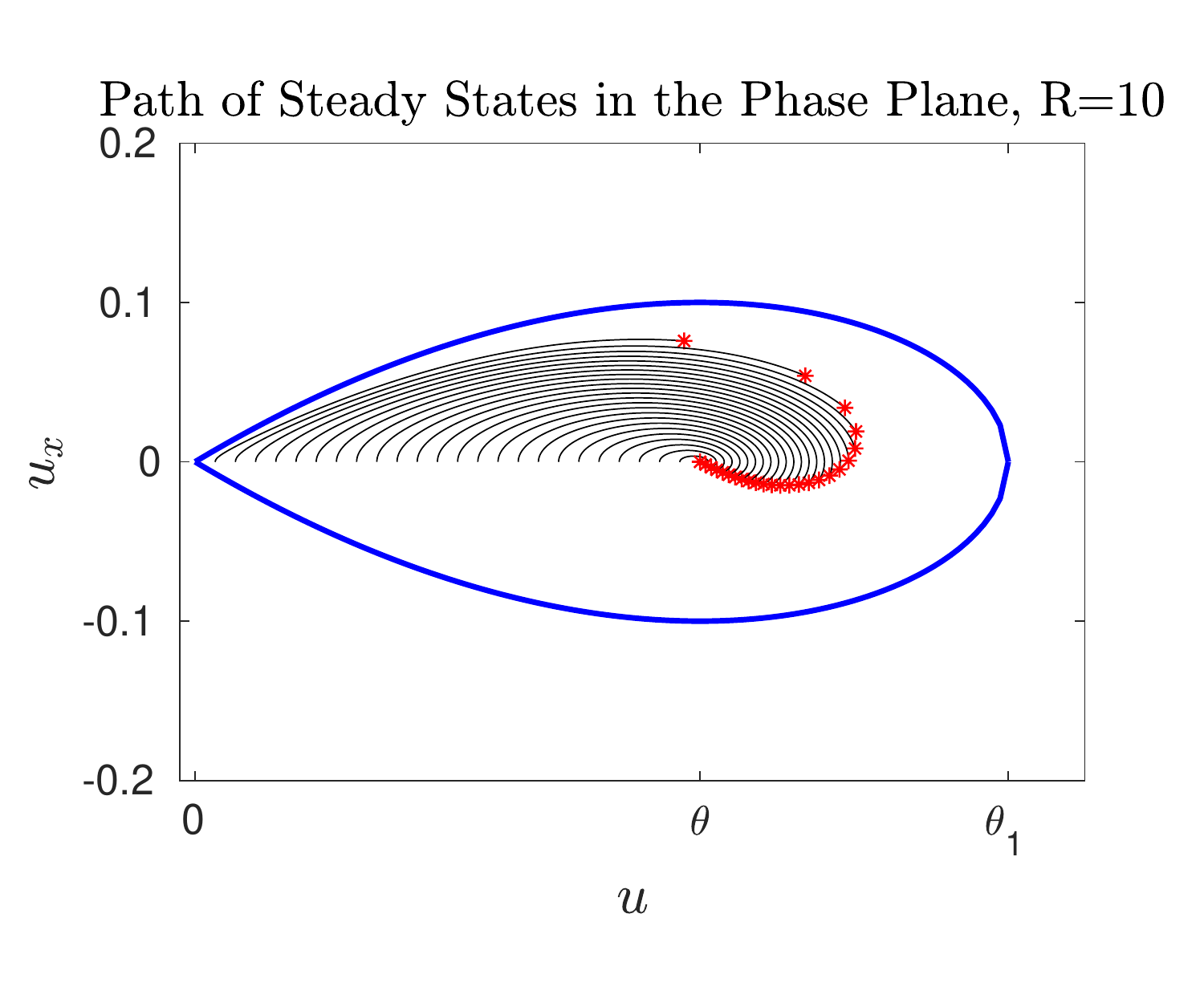}
                    \includegraphics[scale=0.5]{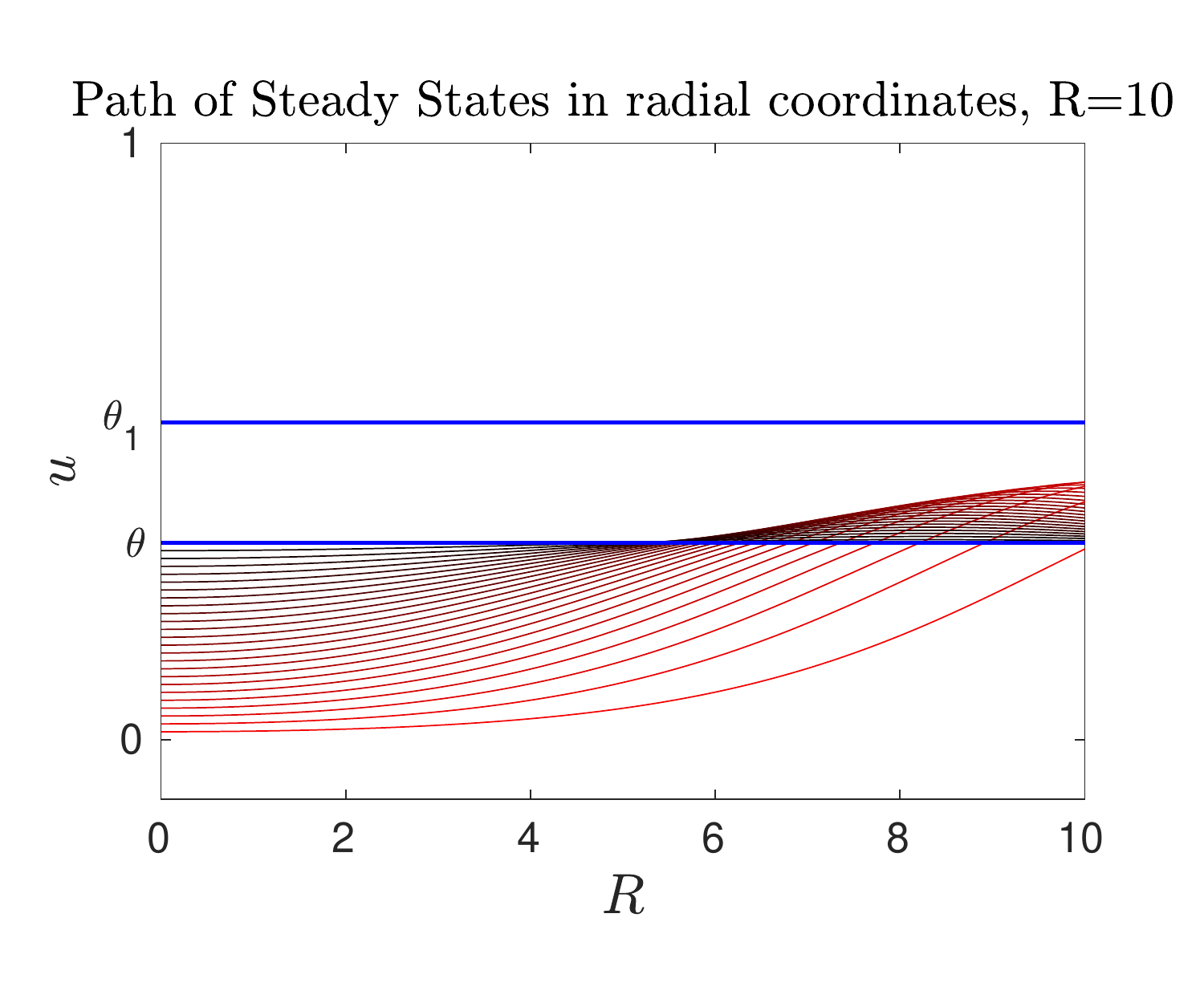}
                    \caption{(Left) in blue the invariant region in the phase space, in black the radial trajectories forming the continuous path of steady-states where the red stars indicate the condition in the boundary. (Right) the corresponding continuous family of steady-states connecting to the stationary solution, $f(s)=s(1-s)(s-1/3)$, seen in radials for $R=10$ and $N=2$.}\label{FigPhRad}
                    \end{center}
                \end{figure}

                3. Restriction to the original domain $\Omega$. Once we have the path of steady-states for any ball in $\mathbb{R}^N$, we restrict our family of steady-states on $\Omega$.

                In this way we obtain the path for any $\Omega$ bounded and any $\mu$.
            \end{proof}
    
    Claim \ref{constr} has constructed a path from $w\equiv 0$ to $w\equiv \theta$, now the concern is how can we control to some part of this path provided that the initial data $u_0\in\mathcal{A}$. In Claim \ref{Zorn}, we proceed to check that the minimal solutions under certain boundary conditions belong to the path of steady-states:
            
            \begin{claim}\label{Zorn}
            Consider a ball $B_R(x_0)$ such that $\Omega\subset B_R(x_0)$. Consider the following problem on $B_R(x_0)$:
                \begin{equation}\label{claimRAD}
                \begin{dcases}
                -\mu\Delta u=f(u)\quad&x\in B_R(x_0),\\
                u=\epsilon\quad&x\in\partial B_R(x_0),\\
                0<u<\epsilon\quad &x\in B_R(x_0).
            \end{dcases}
            \end{equation}
            Then:
            \begin{itemize}
            \item[a)] the problem \eqref{claimRAD} has a minimal solution.
            \item[b)] Let $a(x)$ be the restriction of such minimal solution to $\Omega$, then: The a minimal solution of:
            \begin{equation}\label{restr}
            \begin{dcases}
                -\mu\Delta u=f(u)\quad&x\in\Omega,\\
                u=a(x)\quad&x\in\partial\Omega,\\
                0<u<1\quad &x\in\Omega,
            \end{dcases}
            \end{equation}
                is radial with respect to $x_0$.
            \end{itemize}
            \end{claim}
            \begin{proof}
            \textcolor{white}{.}\newline

            \begin{itemize}
            \item[a)] The existence of a solution of \eqref{claimRAD} follows from sub and supersolutions, taking $\underline{u}=0$ and $\overline{u}=\epsilon$ respectively.
            
            The existence of a minimal solution is proved noticing that if two solutions $u_1$ and $u_2$ of \eqref{claimRAD} cross $\phi(x)=\min\{u_1(x),u_2(x)\}$ is a supersolution of \eqref{claimRAD}, since $\underline{u}=0$ is always a subsolution we have the existence of $u$ a radial solution below $u_1$ and $u_2$.
            A minimal solution $u$ of \eqref{claimRAD} exists by the Zorn's Lemma.
            \item[b)] By contradiction, assume that the minimal solution $v$ of \eqref{restr},  is not the restriction of the minimal solution $u$ of \eqref{claimRAD} then $\psi(x)=\min\{u(x),v(x)\}$ is a supersolution for \eqref{claimRAD} and this will contradict the fact that $u$ is minimal.
            \end{itemize}
            \end{proof}
            
            Since $u_0\in\mathcal{A}$, it exists a time $t_1$ (depending on $u_0$) such that the solution at $t_1$ is below the minimal solution $v$ of Claim \ref{Zorn}, $u(t_1)<v$.
            
            At this point, we set boundary $a_v(x)$ corresponding to the boundary value of the minimal solution of \ref{Zorn}.We will approach the set of steady-states with boundary value $a_v(x)$. Since there exists a minimal solution, by the parabolic comparison principle, we will converge to it. There exists a time $t_2>t_1$ such that the solution $u(t_2)$ will be very close to $v$ in the $C_0(\Omega)\cap H^1_0(\Omega)$ norm, then we apply controllability to attach this minimal solution. The results in \cite[Lemma 8.3]{DARIO} ensure that we do not violate the constraints in this process. The application of the staircase method in the path of Claim \ref{constr} ensures that we can reach $w\equiv \theta$ in finite time.
            
            We summarize the control strategy:
            \begin{itemize}
            \item[-] $a(x,t)=0$ from $t\in[0,t_1]$, where $t_1$ is the time needed until the solution $u$ is below a minimal solution $v$ that belongs to the admissible path.
            \item[-] $a(x,t)=a_v(x)$ from $t\in(t_1,t_2]$, where $a_v(x)$ is as Claim \ref{Zorn} and where $t_2$ is the time needed until the solution $u$ is close enough to the minimal solution $v$ for applying local controllability without violating the constraints.
            \item[-] $a(x,t)$ resulting from local controllability to $v$.
            \item[-] Application of the Staircase method.
            \end{itemize}
            
            Up to here, we have proved the controllability to $w\equiv\theta$ for any $\mu>0$ provided that $u_0\in\mathcal{A}$. Remind that, when $\mu>\mu^*_\theta(\Omega,f)$ the trivial strategy of setting $a=\theta$ plus local controllability also works.
            
            Now we turn our attention to the analysis of the needed controllability times. We start with the case $\mu>\mu^*(\Omega,f)$ (or $\mathcal{A}=L^\infty(\Omega,[0,1])$). In this case, by following the strategy presented above and the comparison principle, the initial datum $u_0\equiv 1$ will give us an upper bound of the controllability time for every $u_0\in L^\infty(\Omega,[0,1])$ to $w\equiv \theta$.

            Here we shall prove the non-uniformity of the control time $T^*_{u_0}$ when $0<\mu\leq\mu^*(\Omega,f)$.
            
            Let us denote by $v_{\min}$ a minimal barrier with respect to the $L^\infty$-norm. Note that, because of the maximum principle, there does not exist any other barrier that is below $v_{\min}$, (it cannot exists a barrier $v$ such that $v_{\min}\leq v>0$).
            The following important claim proves that there exists a perturbation of any $v_{\min}$ such that the corresponding solution converges to $w\equiv 0$ as $t\to\infty$.
            
            \begin{claim}
              Any $v_{\min}$ is not stable and there is a perturbation of $v_{\min}$ that asymptotically goes to $w\equiv 0$.
            \end{claim}
            \begin{proof}
             For $\lambda\in[0,1]$ we consider the problem:
             \begin{equation}
              \begin{cases}
               \partial_tv_\lambda-\mu\Delta v_\lambda=f(v_\lambda)&\quad(x,t)\in\Omega\times(0,T)\\
               v_\lambda(0)=\max\{ v_{\min}-(1-\lambda)\|v_{\min}\|_\infty,0\}\\
               v=0&\quad x\in\partial\Omega
              \end{cases}
             \end{equation}
            by the comparison principle we have that the solution satisfies $v_\lambda(t)>v_{\xi}(t)$ iff $\lambda>\xi$ for all $t\geq 0$. Define the following sets:
            \begin{align*}
             S_0:=\{\lambda\in[0,1] \text{ s.t. } \omega(\lambda v_{\min})=\{w\equiv 0\}\}\\
             S_{v_{\min}}:=\{\lambda\in[0,1] \text{ s.t. } \omega(\lambda v_{\min})=\{v_{\min}\}\}
            \end{align*}

            Observe that the set $S_0$ is an open set.  Take $\lambda+\epsilon$, then for $\epsilon$ small enough there will exist a time $t^*$ for which the solution associated with $\lambda+\epsilon$ is close to $w\equiv 0$ in the $L^\infty$ norm and hence below $w\equiv\theta$. By comparison (or local stability of $w\equiv 0$) we conclude that $\lambda+\epsilon\in S_0$ provided that $\epsilon$ is small enough.
            Then, two situations can arise:
            \begin{itemize}
             \item $S_0\cup S_{v_{\min}}\subsetneq[0,1]$. Then, taking $\lambda \in [0,1]\backslash(S_0\cup S_{v_{\min}})$, we see that it leads to a contradiction with the fact of $v_{\min}$ being a minimal solution. Indeed, the solution approaches the set of steady-states and it cannot approach $w\equiv 0$ since by comparison or local stability it would converge to it. Then it means that there is a sequence of times $t_n$ such that $v_\lambda(t_n)$ converges  some other steady-state. 
             \item $S_0\cup S_{v_{\min}}=[0,1]$. Since $S_0$ is open it means that $S_{v_{\min}}=[\lambda^*,1]$ (the case in which $S_{\min}=\{1\}$ already proves our claim). This also implies the instability of $v_{\min}$. One can see that we can take a sequence of $\epsilon_n\to 0$ and take $\lambda_n=\lambda^*-\epsilon_n$, then, by continuity with respect to the initial data we have that there exists a sequence of times $t_{\epsilon_n}$ such that $\|v_{\lambda}(t_{\epsilon_n})-v_{\min}\|_\infty\to 0$. However, this would be a meta-stability phenomena since the trajectory by assumption will converge to $w\equiv 0$. We have just found a sequence of states tending to $v_{\min}$ by below that asymptotically converge to $w\equiv 0$ as $t\to \infty$.
             
             Up to here, we already proved the instability of $v_{\min}$. However, this last situation cannot arise. Remind that we are approaching $v_{\min}$ in the $L^\infty$ norm, and we have that each initial data corresponding to $\xi\in S_{\min}\backslash\{1\}$ is uniformly away of $v_{\min}$. Then, the sequence found before of $v_\lambda(t_{e_n})$ will be above $v_\xi(0)$ for any $\xi\in S_{\min}$ leading to a contradiction with the maximum principle.
            \end{itemize}

            \end{proof}
            
            Let us now go back to the proof of the unboundedness of $T^*_{u_{0}}$. Assume by contradiction that there exists a finite $T^*$ such that $T^*\geq T_{u_{0}}^*$ for all $u_0\in\mathcal{A}$. Take a sequence of initial data $u_{0,n}$ converging to a minimal solution $v_{\min}$ and the corresponding sequence of controls $a_n$ that bring it to $w\equiv\theta$ in time $T^*$. By Banach-Alaoglu there exist the limit of $a^*$ of $a_n$ in the weak$^*$ topology and since the limit of $\lim_{n}u_{0,n}=v_{\min}$ we obtain a control $a^*$ that brings in finite time to $w\equiv\theta$ a non-controllable initial datum, since $v_{\min}$ is a barrier.

    \subsection{Example: A path that connects with the minimal barrier}\label{pathexample}
    
    If we consider $\Omega$ a ball and $0<\mu\leq \mu^*(\Omega,f)$, a nontrivial steady solution with boundary $a=0$ exists. In this geometry, we can construct following the same arguments, a path that connects $w\equiv0$ with the minimal barrier with respect to the $L^\infty$-norm (see Figure \ref{tobarrier}). This allows us to use the staircase method iteratively to obtain an open-loop stabilization, provided that $u_0\in\mathcal{A}$. Of course, we will never reach this target in a finite time horizon since the trace of this path is not bounded away from $0$.
    
    One can prove that the uncontrollability in a finite time horizon to a barrier does not depend on the control strategy. This can be done using a duality argument similar to the one in \cite{DARIO}, where a waiting time phenomenon was observed for the heat equation with nonnegative controls. However, the details will be developed in a future article.
    
    \begin{figure}[H]
    \centering
     \includegraphics[scale=0.5]{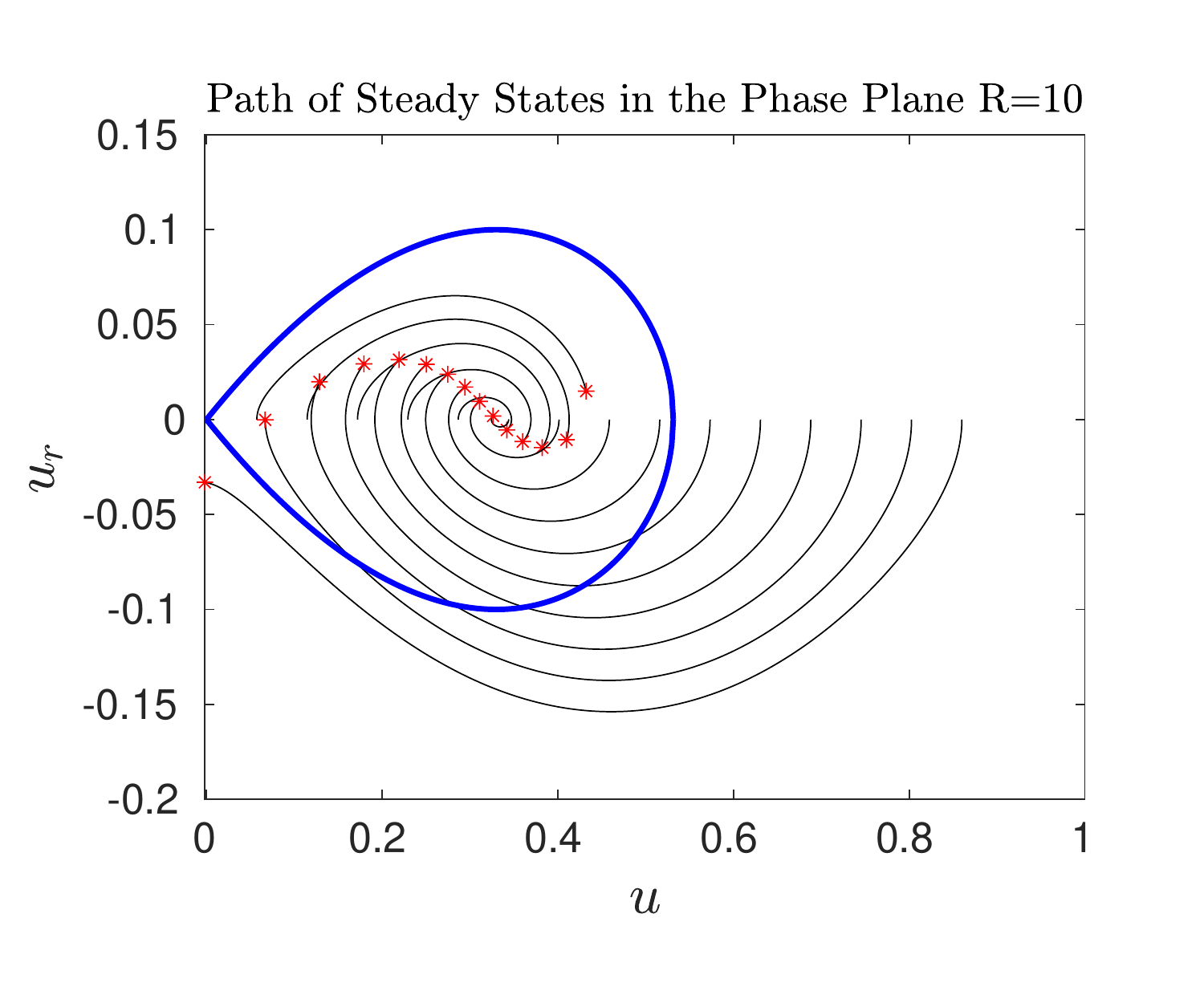}
     \includegraphics[scale=0.5]{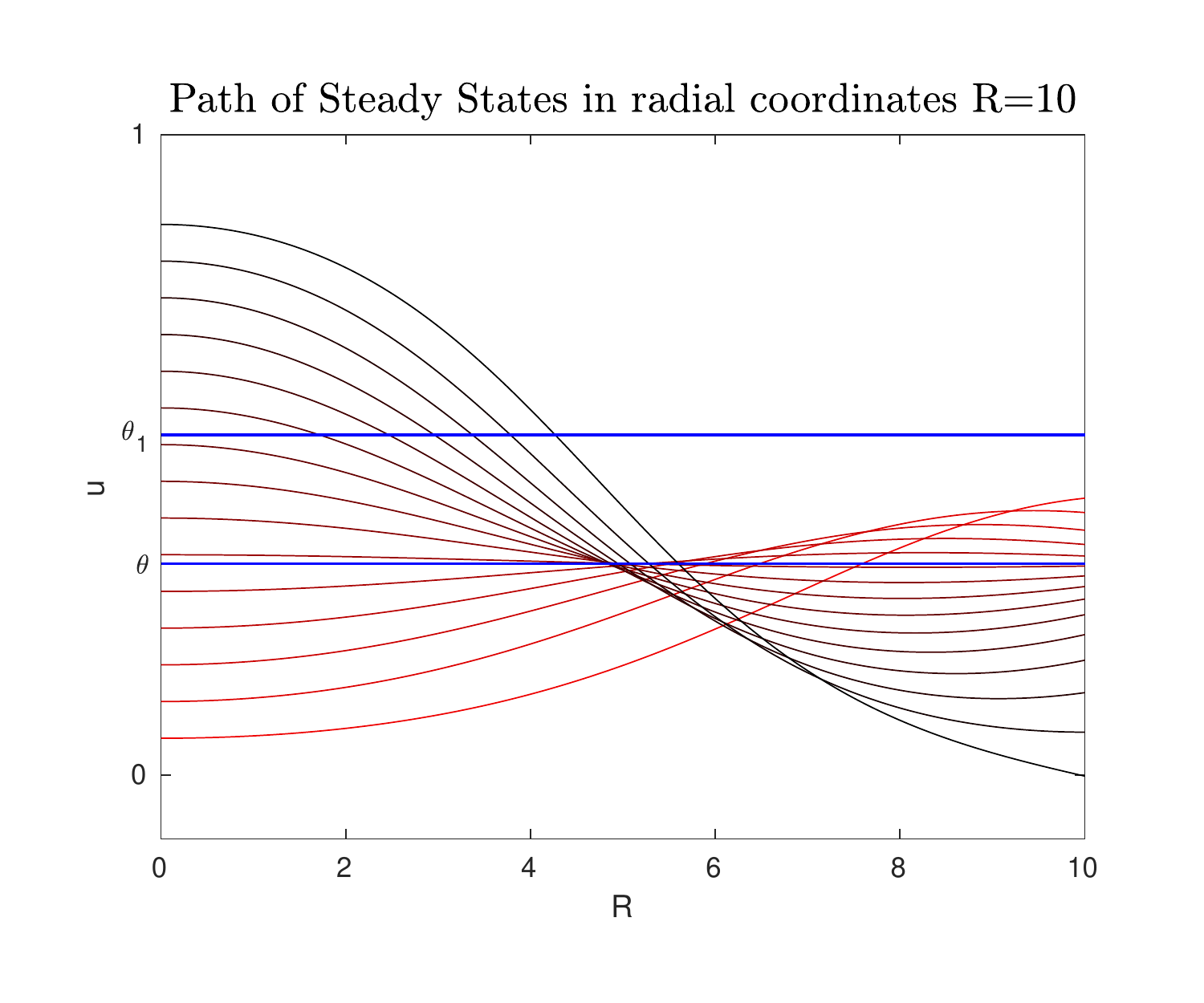}
    \caption{(Left) in blue the invariant region in the phase space, in black the radial trajectories forming the continuous path of steady-states where the red stars indicate the condition in the boundary. (Right) the corresponding continuous family of steady-states connecting to the minimal barrier, $f(s)=s(1-s)(s-1/3)$, seen in radials for $R=10$ and $N=2$.}\label{tobarrier}
    \end{figure}

    \subsection{The set $\mathcal{A}$}\label{Aexample}
            In this section, we will give more examples of initial data that belong to $\mathcal{A}$. To do it we use the path of steady-states that connects $w\equiv 0$ with the minimal barrier with respect to the $L^\infty$-norm $v_{\min}$ in the ball.
            
            \begin{itemize}
             \item Do to Remark \ref{critRemark}, we know that any element in any admissible path that starts at $w\equiv 0$ belongs to $\mathcal{A}$.
             \item We can consider balls of radius $R$ and center $x_0$ that include the domain $\Omega$ and use the elements of the paths that go from $w\equiv0$ to their minimal solution as supersolutions of the original problem. Since these elements go to zero (Remark \ref{critRemark}), any initial datum below them will also go to zero in the domain $\Omega$.

             \end{itemize}

    \subsection{Proof of the Theorem 1.3, $F(1)=0$ case.}
    In this case the scheme is very similar with only one difference, that the traveling waves for the Cauchy problem are stationary. In particular this will imply the following claim:
        \begin{claim}\label{newclaim}
        If $F(1)=0$, there is a unique solution of:
        \begin{equation}\label{ohhh}
        \begin{dcases}
        -\Delta v=f(v)&\quad x\in\Omega,\\
        v=a&\quad x\in\partial\Omega,
        \end{dcases}
        \end{equation}
        for $a\in\{0,1\}$
        \end{claim}
        Hence Claim \ref{newclaim} ensures that there is no barrier. The construction of the path of steady-states works very similarly also due to the fact that the traveling waves will generate an invariant region of negative energy where $w\equiv\theta$ will be inside and $w\equiv1$ and $w\equiv0$ will be at the boundary of the region. Moreover, the restriction of the traveling waves gives other admissible paths of continuous steady-states for going from $w\equiv 0$ to $w\equiv 1$ and vice-versa.
        We proceed to prove the claim:
        \begin{proof}
        The proof follows by contradiction. By simplicity, assume $a=0$, the other case follows with the same argument. Assume it exists a solution of \eqref{ohhh}. Extend the solution by $0$ in a ball $B_R$ for $R$ big enough such that $\Omega\subset B_R$. Since $w\equiv 1$ is a supersolution, we have that in the ball $B_R$, there exists a nontrivial solution with boundary value $0$. This solution is radial since we are working in a ball, and we can write it as a radial ODE. We see that the boundary of this solution must satisfy that $v_r^2>0$, which implies that the energy at $R$ is positive. By the dissipation of the radial ODE, we know that the energy at the origin of the ODE (or center of the ball) is bigger or equal than this energy. From here, it leads to contradiction because the ODE cannot cross the horizontal axis of the phase plane $v_r=0$, between $0$ and $1$ because it is a region of negative energy.
        
        The uniformity of the control time follows from the convergence to $w\equiv 0$ for the initial datum $u_0\equiv 1$
        
        \end{proof}


\section{Numerical illustration}\label{Snum}

\subsection{Minimal controllability time}
The fact of adding state constraints to the problem
induces a minimum controllability time (see \cite{DARIO,LOHEAC}). The staircase method provides the controllability in time large but it is far from being the minimal time. In this subsection we are going to compute it numerically.

We will minimize the time of exact controllability $T$:
\begin{equation*}
 I[a]=T,
\end{equation*}
 under the dynamic constraints
\begin{equation*}
\begin{dcases}
  u_r-\mu u_{rr}-\mu \frac{N-1}{r}u_r=u(1-u)(u-\theta),\\
 u(t,R)=a(t),\\
 u_r(t,0)=0,\\
 u(0,r)=0,
\end{dcases}
\end{equation*}
and:
\begin{align*}
 &0\leq u(t,r)\leq1\quad \forall (t,r)\in [0,T]\times [0,R],\\
 &\theta-\epsilon \leq u(T,r)\leq \theta+\epsilon \quad \forall r\in[0,R].
\end{align*}
For the numerical scheme, finite differences where employed using an implicit
scheme with a fixed point for the nonlinearity.
In Figure \ref{ControlMT} one can see the control function showing a bang-bang behavior.

\begin{figure}
    \begin{center}

 \includegraphics[scale=0.5]{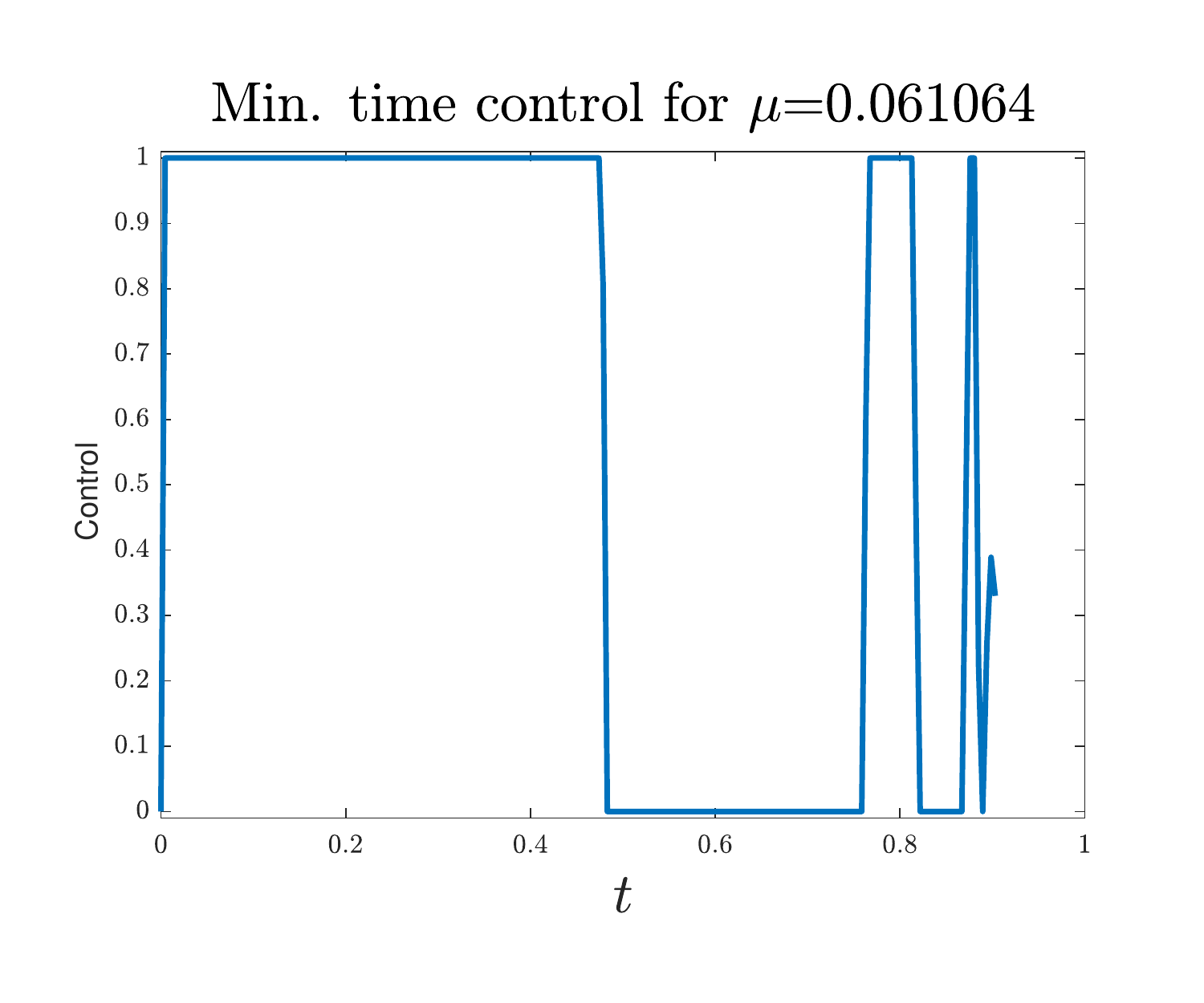}
  \includegraphics[scale=0.5]{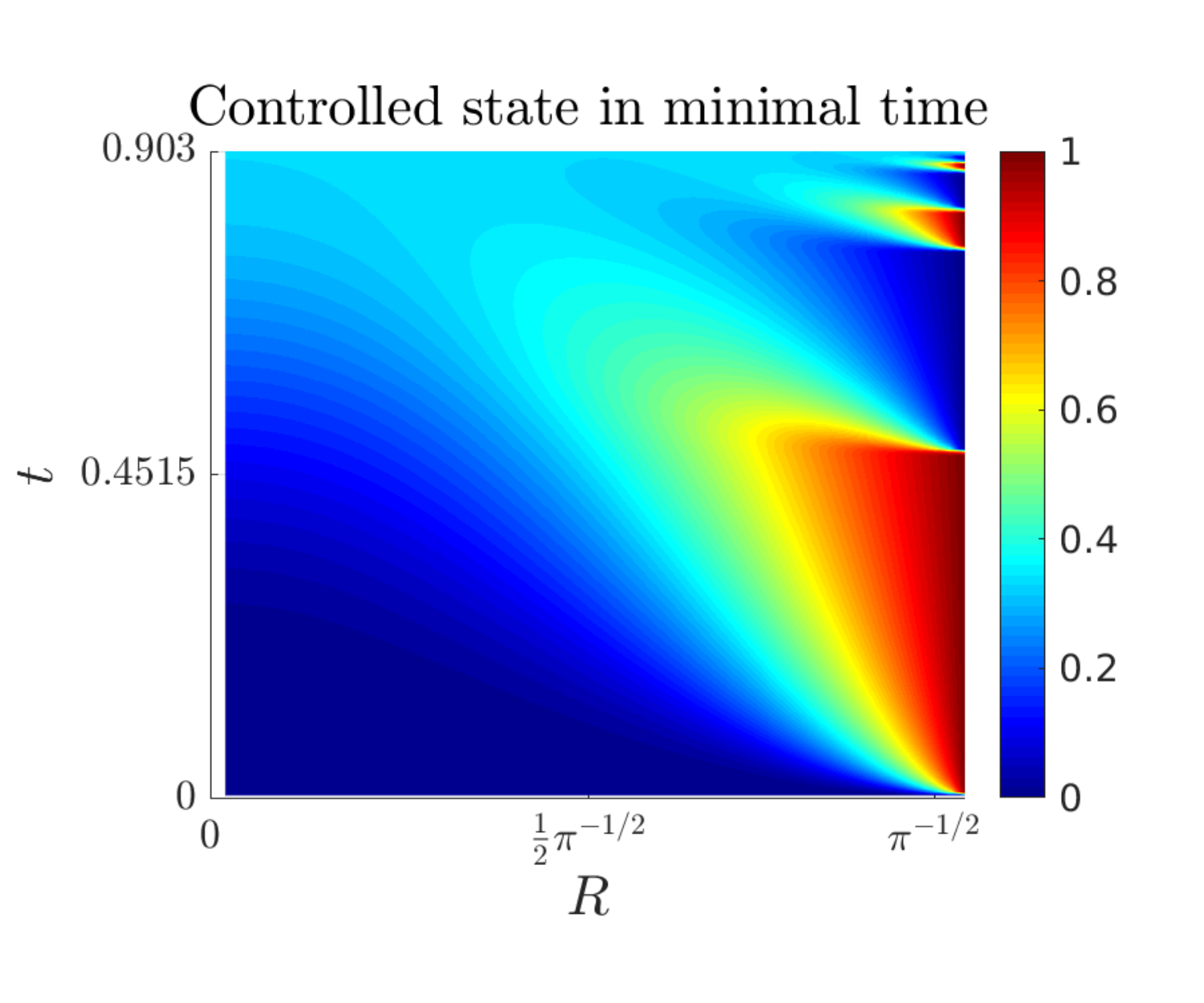}
\caption{Minimal controllability time from $u_0\equiv 0$ to $u(T_{\min})=w\equiv \theta$ in a ball of measure $1$ with diffusivity $\mu=0.0611$, $\epsilon=0.01$. Nonlinearity $f(s)=s(1-s)(s-1/3)$}\label{ControlMT}
\end{center}
\end{figure}

The numerical simulations point that the control in minimal time is bang-bang.

\subsection{Visualization of the continuous path for large R}
 Figure \ref{Illustration} shows several captures of the continuous path of steady-states for $R=30$.
\begin{figure}
    \begin{center}

\includegraphics[scale=0.25]{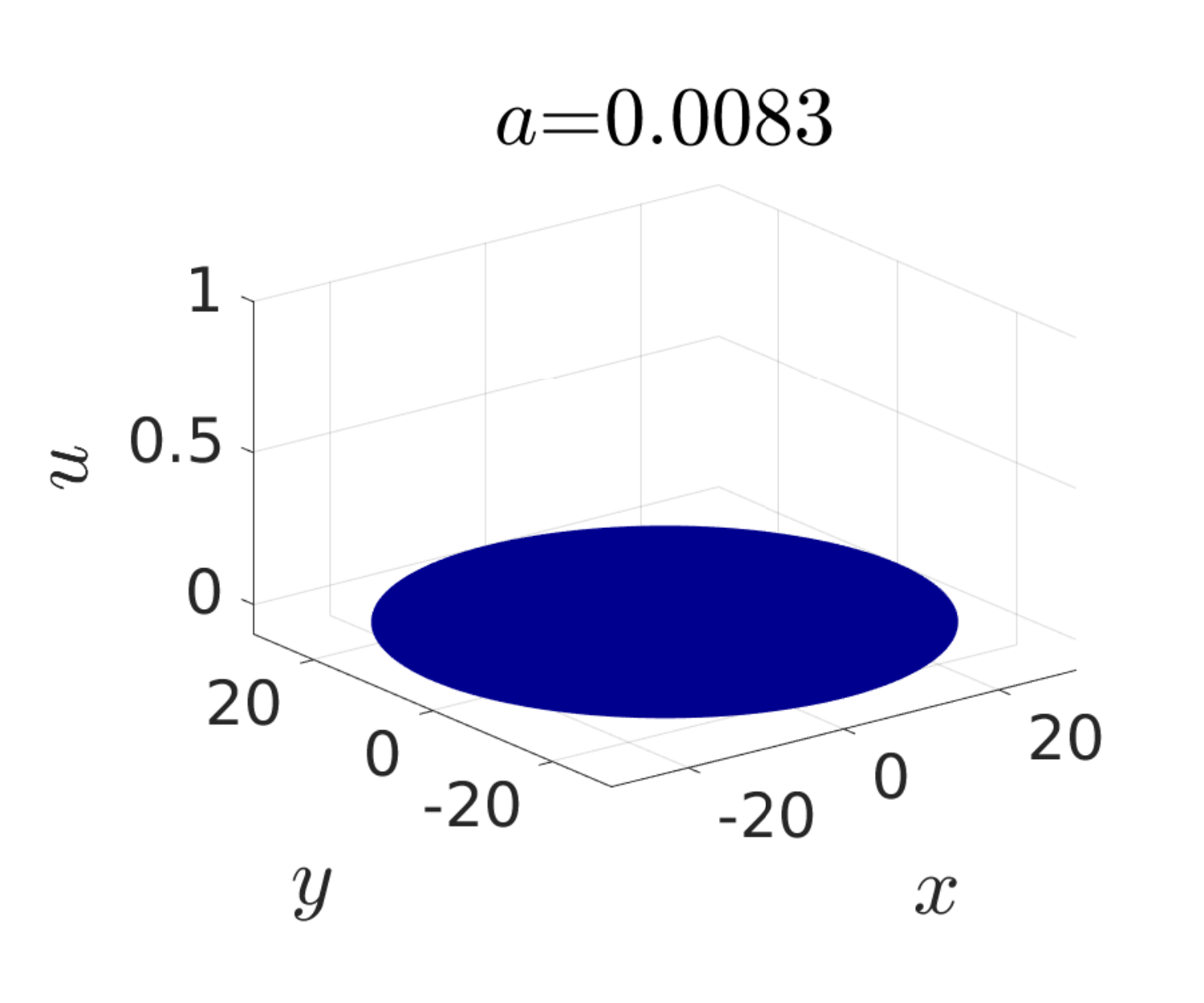}
\includegraphics[scale=0.25]{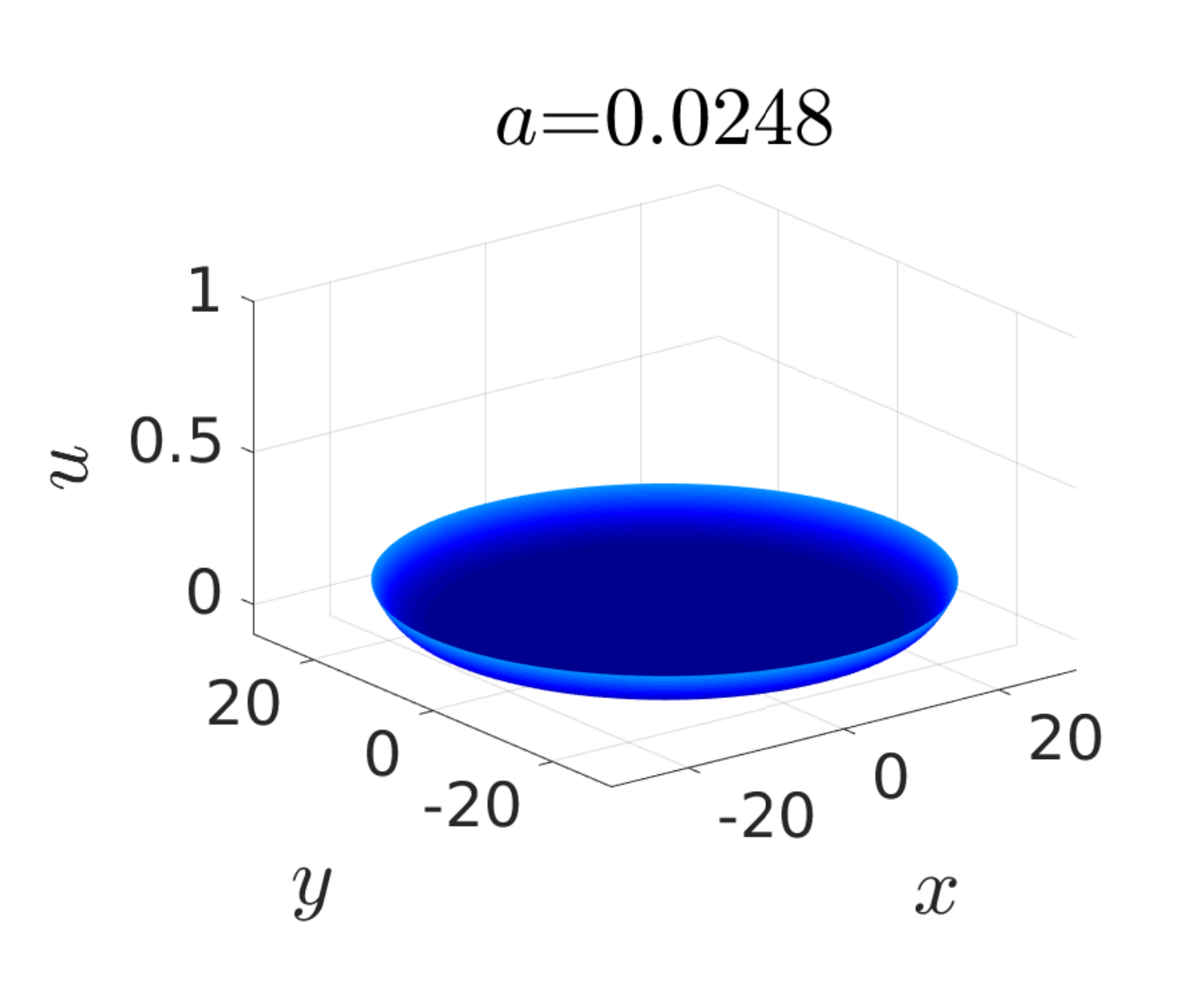}
\includegraphics[scale=0.25]{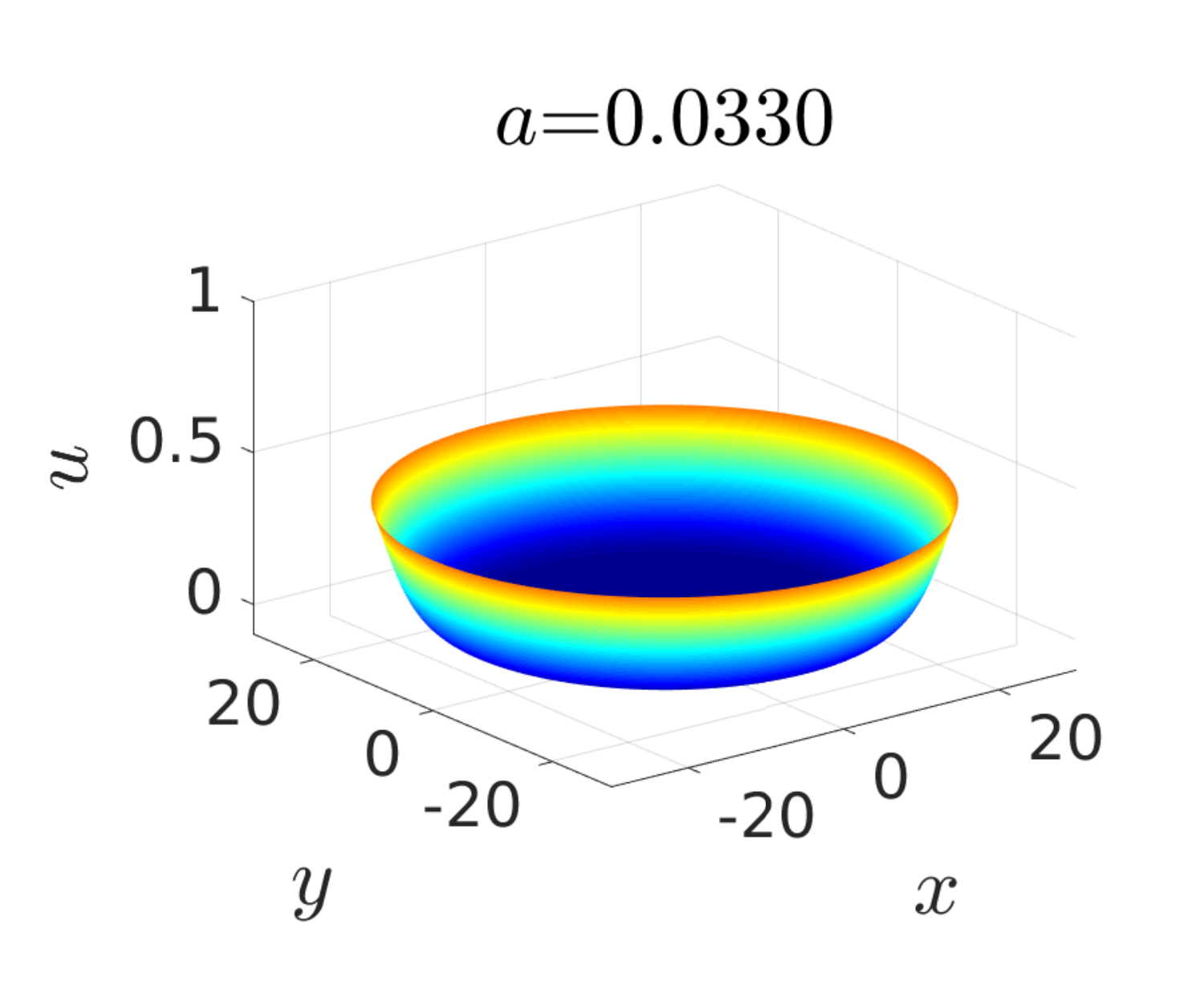}
\includegraphics[scale=0.25]{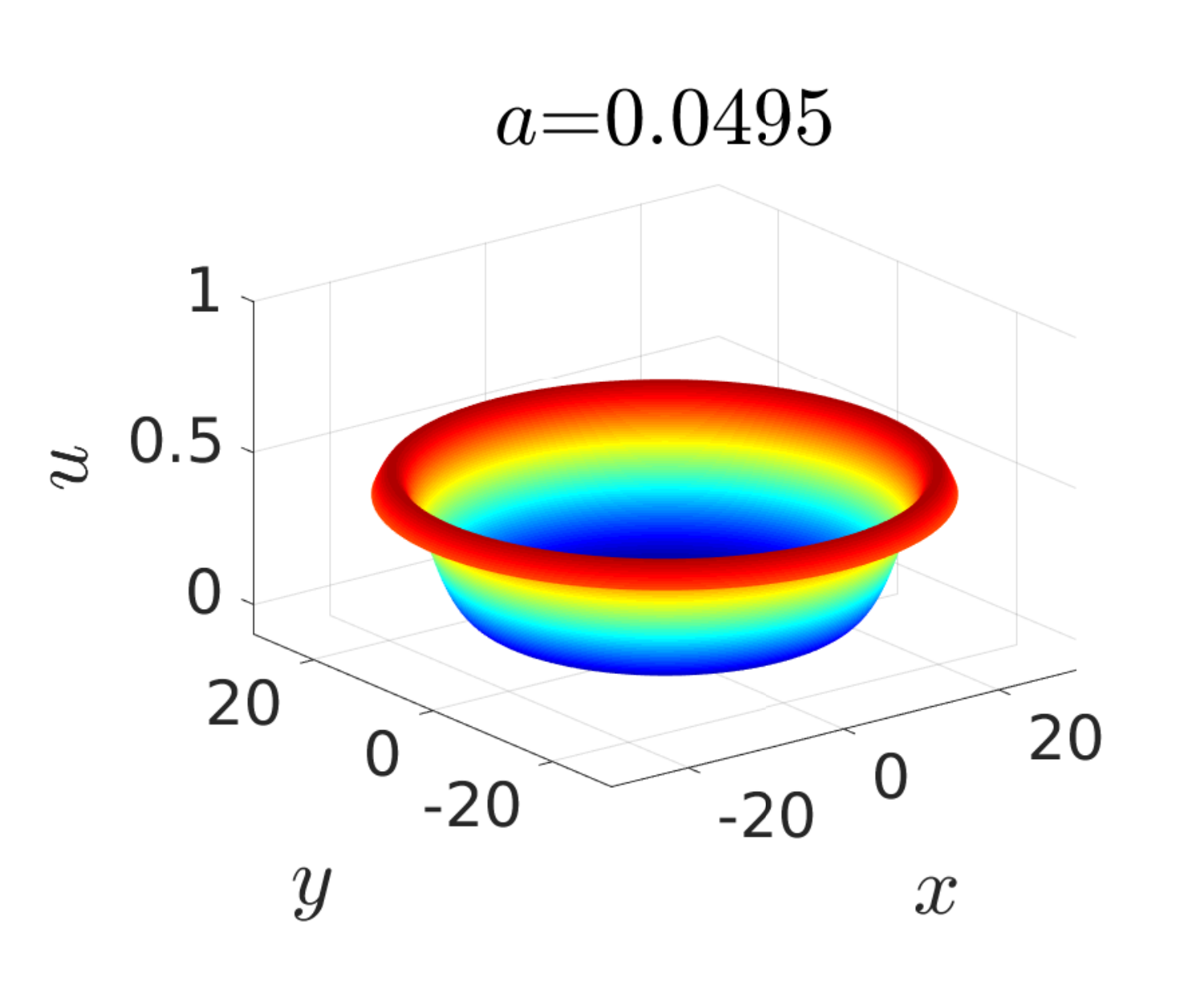}
\includegraphics[scale=0.25]{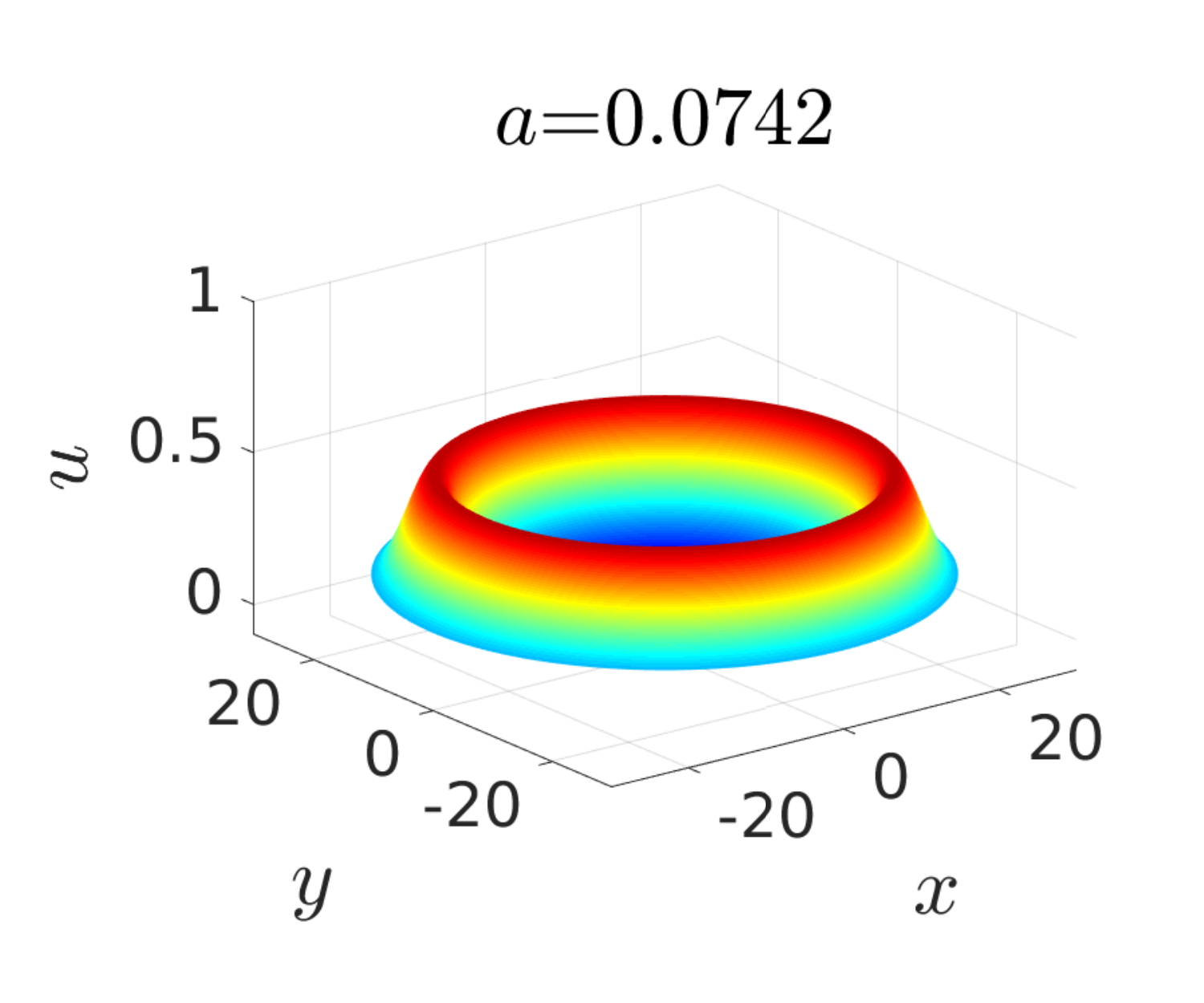}
\includegraphics[scale=0.25]{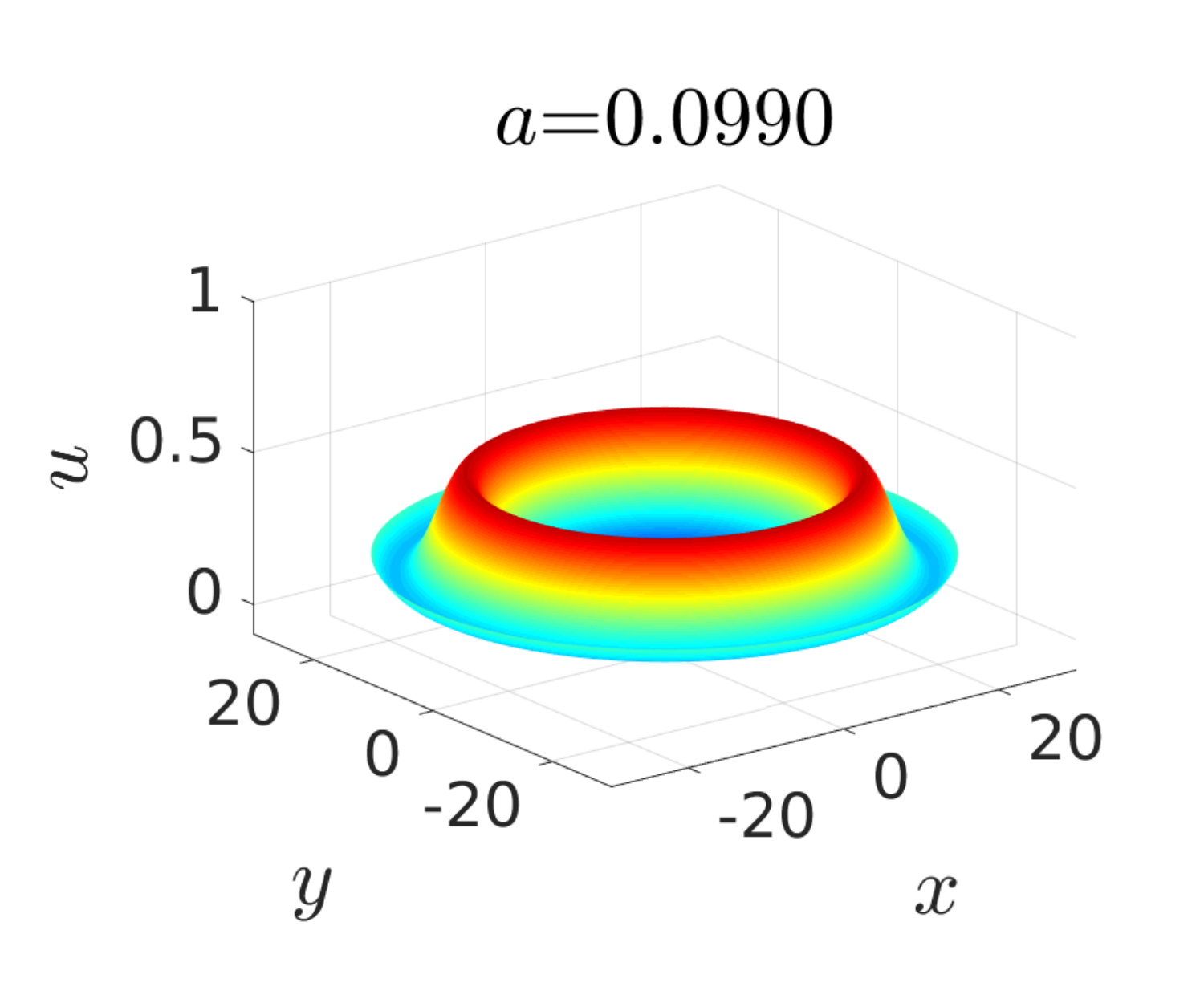}
\includegraphics[scale=0.25]{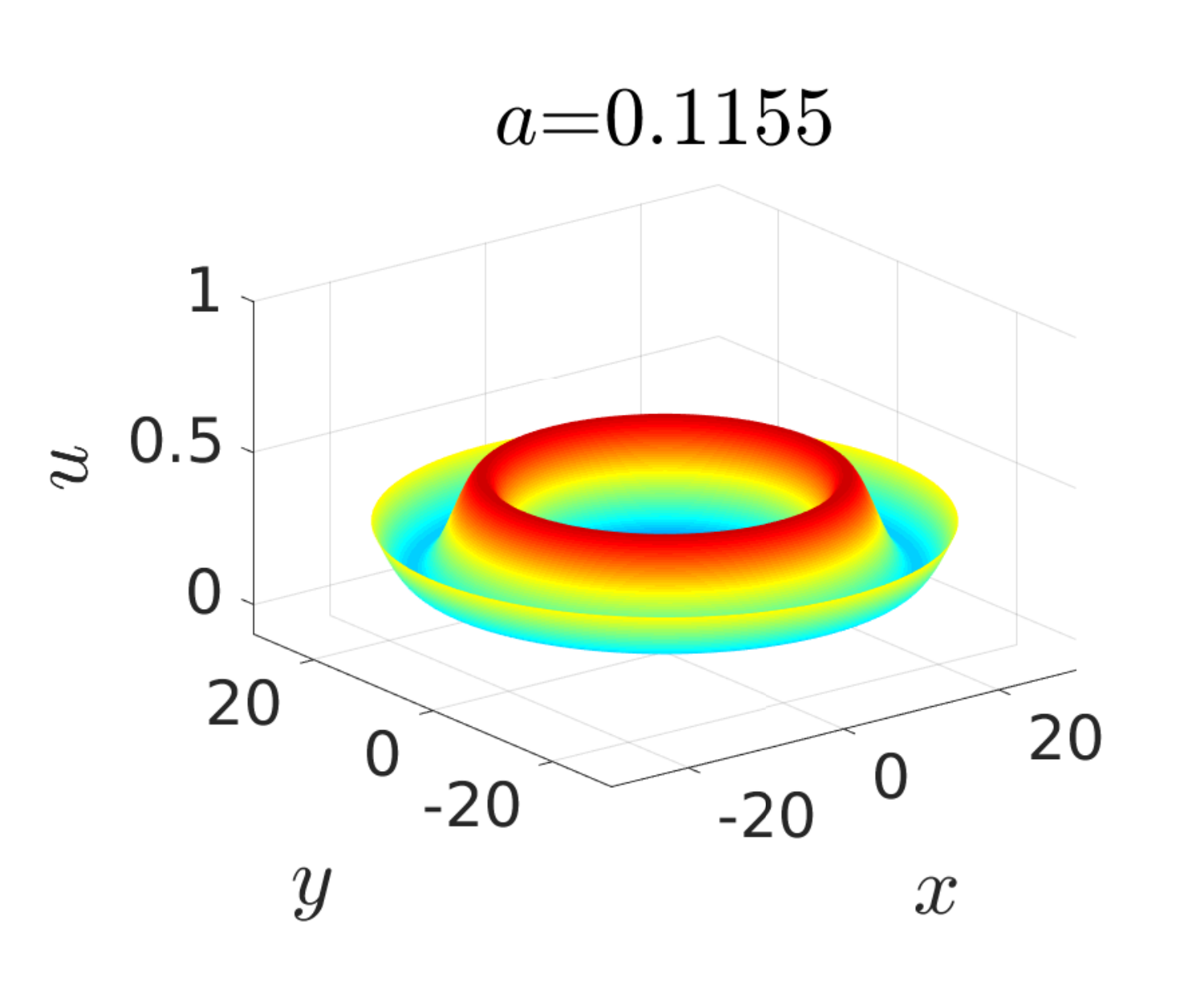}
\includegraphics[scale=0.25]{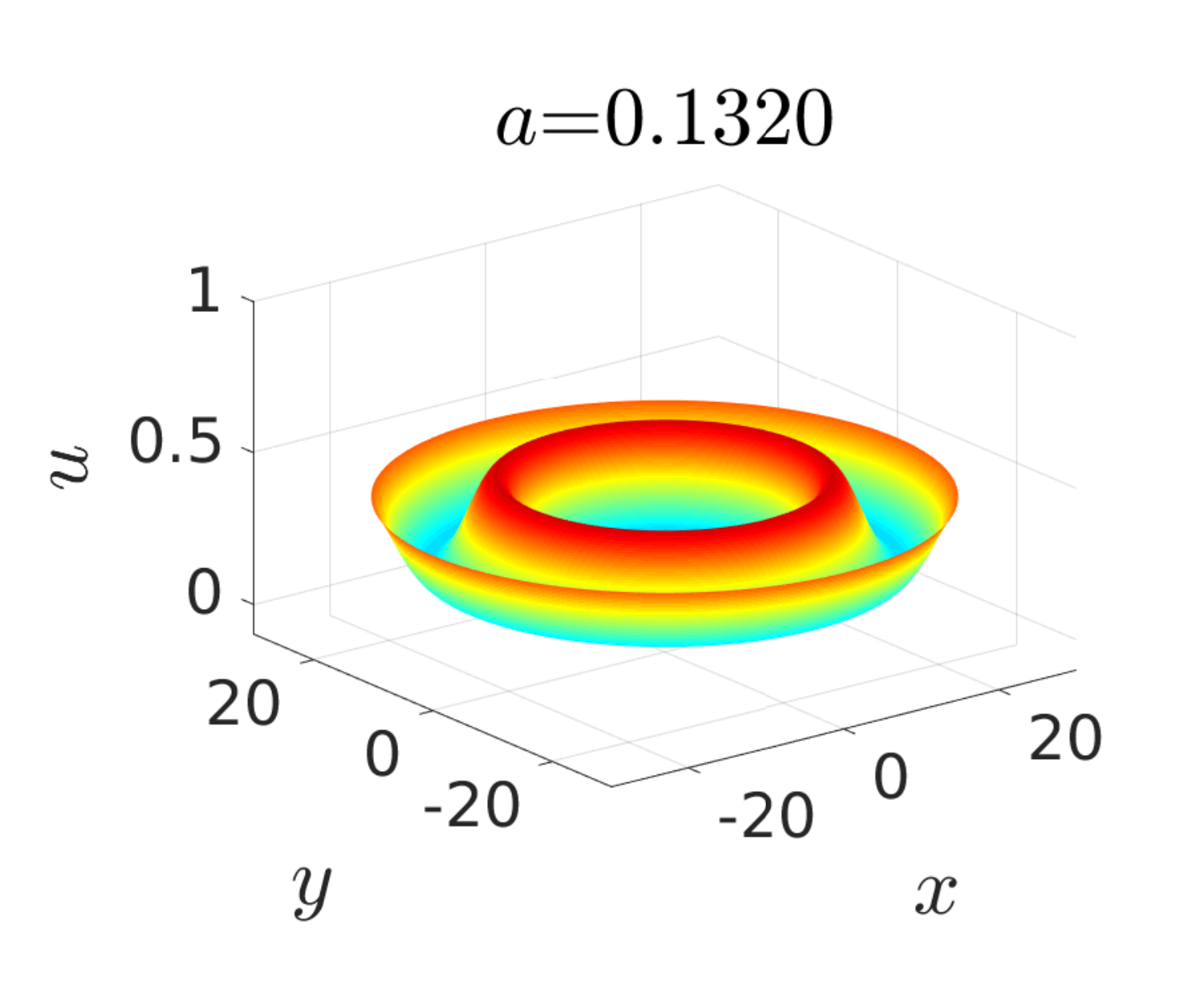}
\includegraphics[scale=0.25]{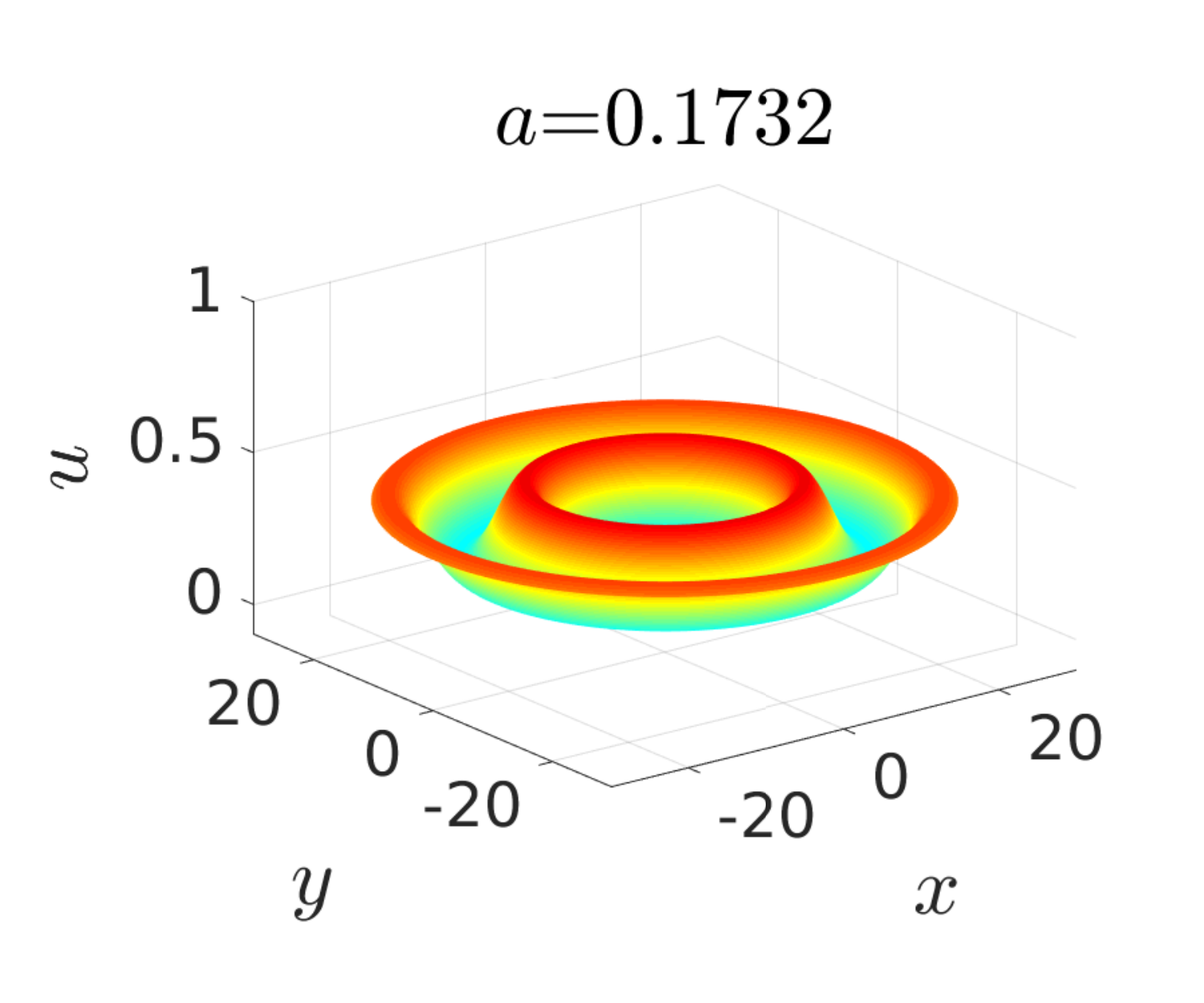}
\includegraphics[scale=0.25]{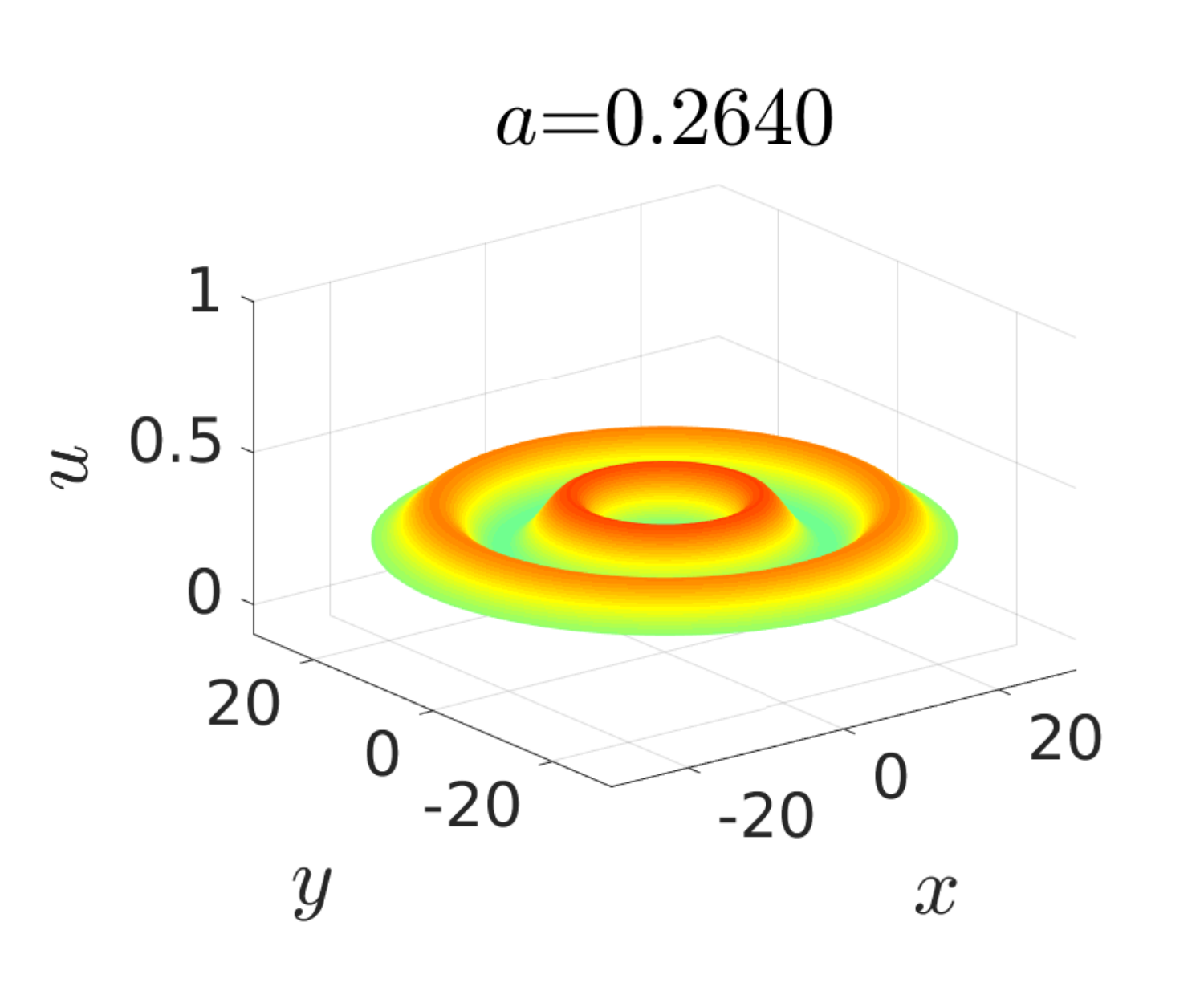}
\includegraphics[scale=0.25]{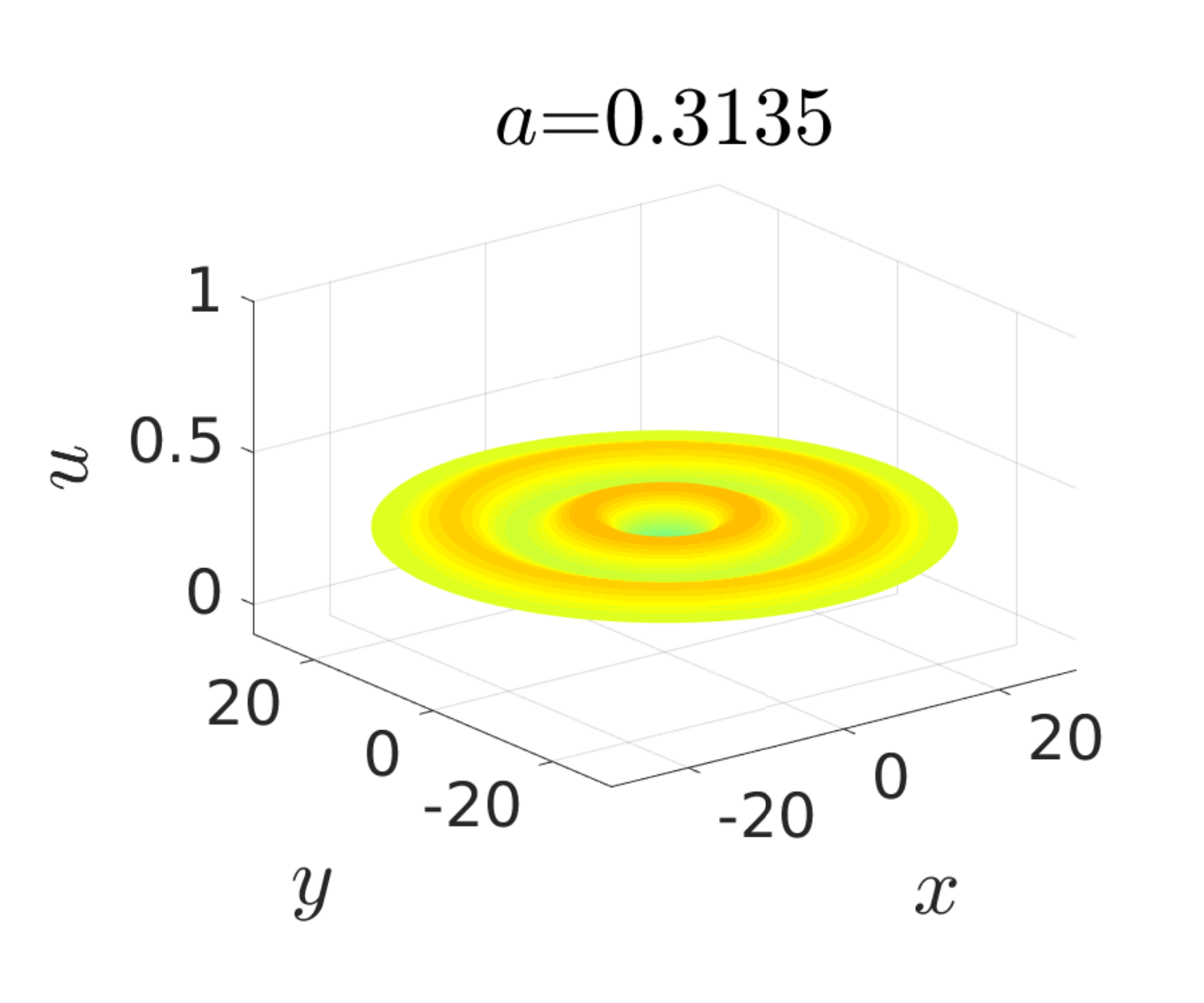}
\includegraphics[scale=0.25]{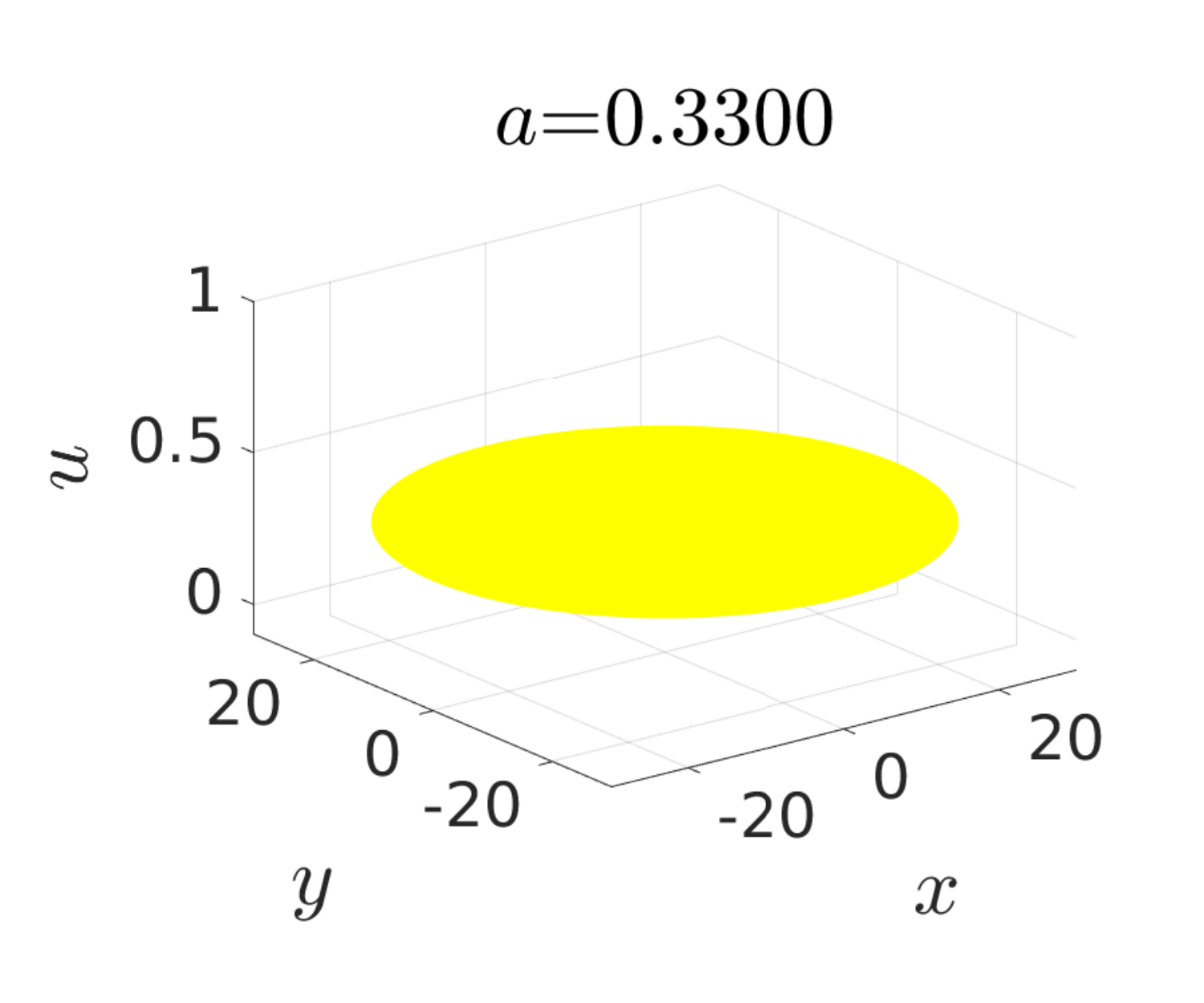}

  \caption{Some steady-states forming part of the continuous path of admissible steady-states connecting to $u=\theta=0.33$ for $R=30$ and $N=2$. Nonlinearity $f(s)=s(1-s)(s-\theta)$.}\label{Illustration}
  \end{center}
\end{figure}
In Figure \ref{Trace}, the trace of the continuous path is shown. One can observe its oscillations which also are natural in the 1-d case. Notice that the appearance of nontrivial solutions with boundary value $\theta$ by the comparison principle can create barriers so that the monotonous path cannot work. In the one dimensional case the oscillations are higher due to the lack of dissipativity of the ODE system generating the elliptic solutions. Indeed, locally around $\theta$, one can observe in the 1D case the harmonic oscillator. 
\begin{figure}
    \begin{center}

  \includegraphics[scale=0.5]{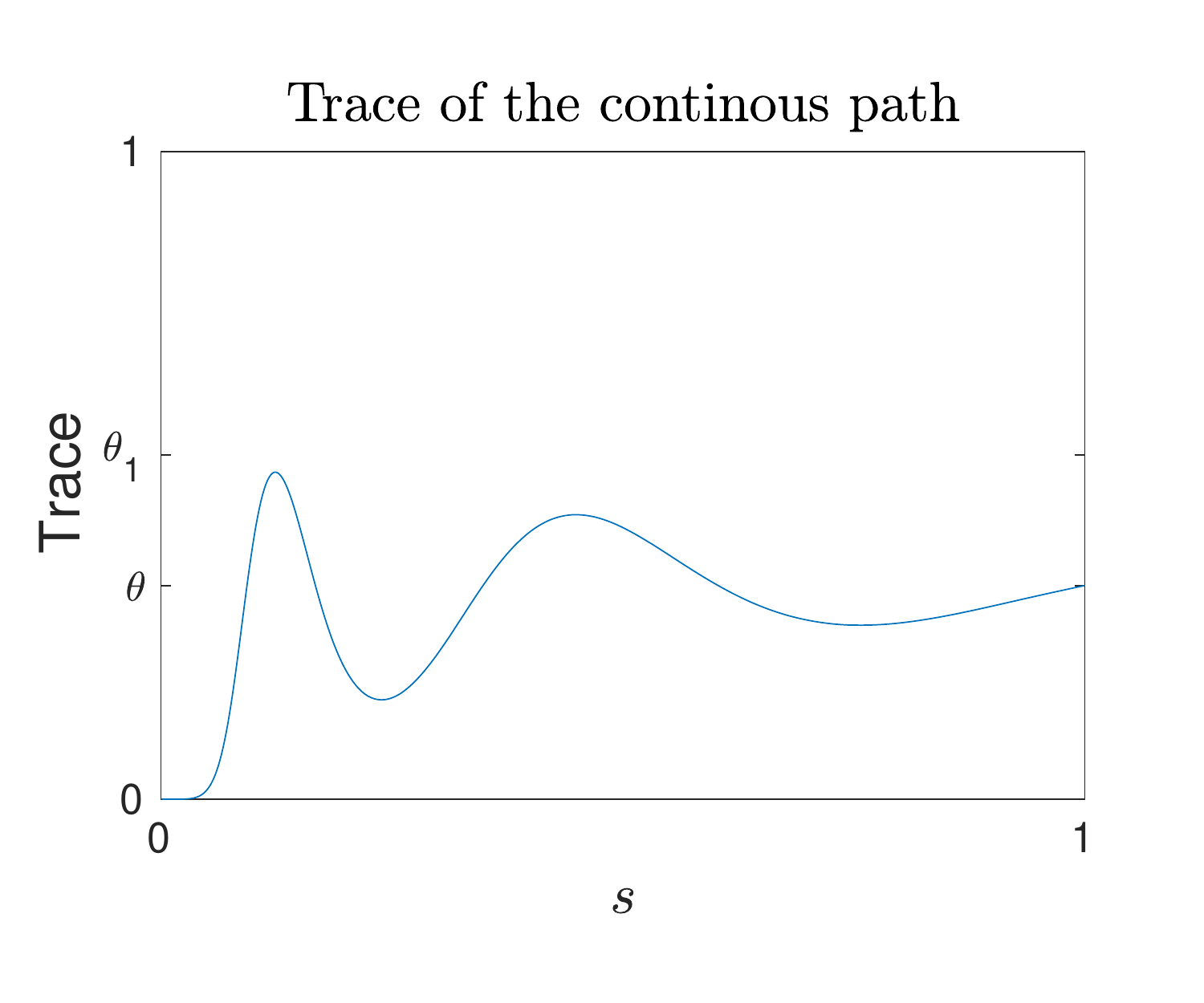}
  \caption{Trace of the continuous path of admissible steady-states.}\label{Trace}
  \end{center}
\end{figure}

\subsection{Quasistatic control}

In this subsection numerical approximations in IPOPT are performed in order
to observe an approximation of the quasistatic control strategy used in
\cite{CORON-TRELAT} in which the authors stabilize the resulting
trajectory of setting the boundary as the trace of the path of
steady-states.

The discrete version of the following cost functional is minimized
\begin{equation*}
 I[a]=\int_0^T a_t(t)^2 dt,
\end{equation*}
 under the dynamic constraints
\begin{equation*}
\begin{dcases}
  u_r-u_{rr}-\frac{N-1}{r}u_r=u(1-u)(u-\theta),\\
 u(t,R)=a(t),\\
 u_r(t,0)=0,\\
 u(0,r)=0,
\end{dcases}
\end{equation*}
and:
\begin{align*}
 &0\leq u(t,r)\leq1\quad \forall (t,r)\in [0,T]\times [0,R],\\
 &\theta-\epsilon \leq u(T,r)\leq \theta+\epsilon \quad \forall r\in[0,R],\\
 &\|a_t\|_\infty<\epsilon,
 \end{align*}
where the last constraint can be required since we already know that the system is controllable to $\theta$ if we start with an initial datum small enough.

Figure \ref{state-phase} shows the controlled state against the time and also in the phase space, in the later one elliptic solutions are also shown.
\begin{figure}
    \begin{center}
  \includegraphics[scale=0.4]{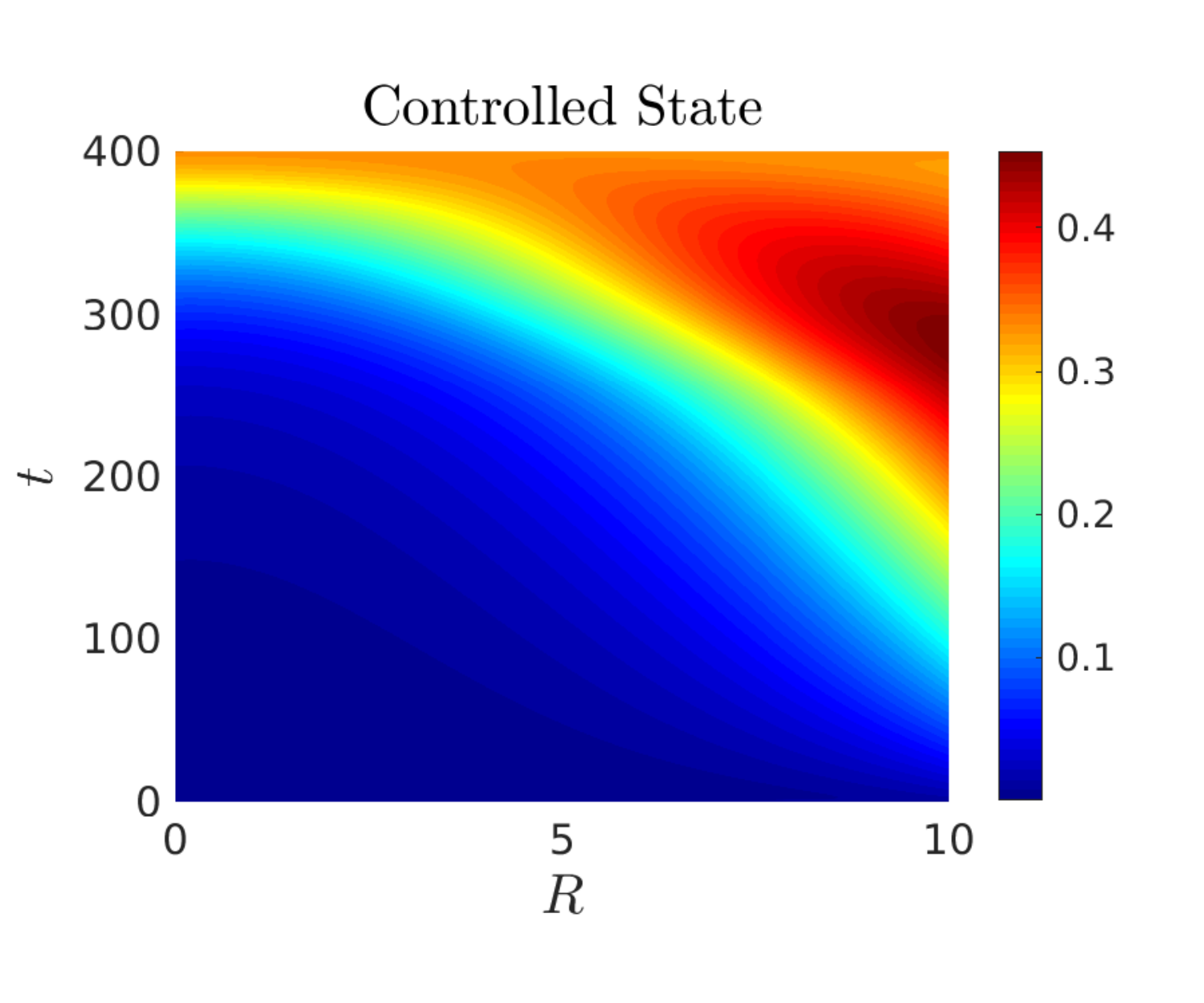}
    \includegraphics[scale=0.4]{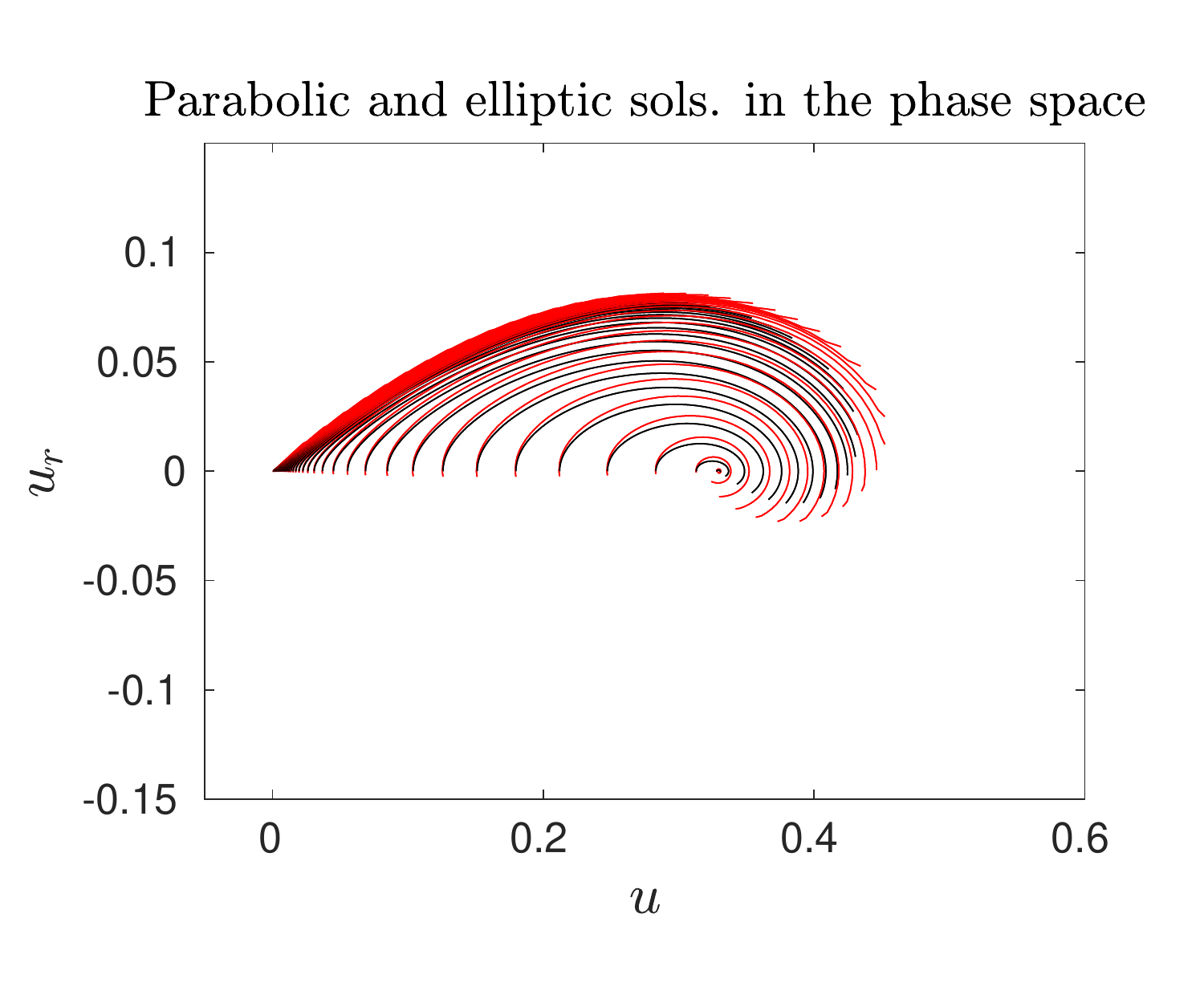}
    \caption{ Time evolution of the controlled state in radial coordinates. (left).  In red the parabolic controlled state at different
    equispaced times in the phase plane.
    In Black the elliptic solution which has the same condition
    at the origin (right). Nonlinearity $f(s)=s(1-s)(s-\theta)$.}\label{state-phase}
    \end{center}
\end{figure}

\section{Conclusions and perspectives}\label{Sconcl}
In this work, we have seen that the presence of state constraints can lead to the existence of nontrivial solutions that act as fundamental obstructions to the controllability for certain initial data (barrier functions). The domain and the nonlinearity play an essential role in the existence of barriers for reaching $w\equiv 0$.

Moreover, the existence of traveling waves for the corresponding Cauchy problem in the whole space helps us to determine the nonexistence of barrier functions for reaching $w\equiv 1$.

In the bistable case, for reaching the intermediate equilibrium $w\equiv \theta$, the staircase method has been crucial. The construction of the corresponding path relies on two ideas. First, we enlarge the domain to a ball, and second, when the elliptic problem in the ball is understood as an ODE problem, we observe that there exists a positively invariant region in the phase space containing our target.

When constraints in the state are present, we can encounter different situations, some of the analyzed in the paper that are summarized hereafter:
\begin{enumerate}
 \item There does not exist any continuous path of admissible steady-states connecting the initial steady-state and the target. However, we are able to control from one to another because the target is an attractor . This is the case discussed employing a comparison with the traveling waves with the stabilization to $w\equiv1$. 
 \item There does not exist any continuous path of admissible steady-states connecting an initial steady-state with the target, and we are not able to control. The emergence of a barrier illustrates this case. 
 \item There is a continuous path of admissible steady-states from our initial steady-state and the target. This implies the controllability. Moreover, we emphasize that:
 \begin{itemize}
 \item For such domain, there can be nontrivial solutions that can act as a barrier for certain initial data. However, the fact that we have an attainable path for our initial data ensures that we will be able to control regardless of its existence.
 \item The stability of the target steady-state does not matter. The path of steady-states ensures that we can control towards an unstable equilibrium. 
\end{itemize}
\end{enumerate}
Moreover, we have observed that the staircase method can be used to stabilize, in an open-loop manner, around a barrier.

The construction of paths of admissible steady-states in general is not a trivial issue, however, we have seen in Remark \ref{critRemark} that we can not aim to build a path along two steady-states if their $\omega$-limits with control $a=0$ are in comparison.

We can summarize our conclusions for the Bistable case with the following diagrams of Figure \ref{maptetasmall} and \ref{mapteta12}, where the possible transitions to the constant steady-states is depicted.
\begin{figure}[H]
\centering
 \begin{tikzcd}[column sep=tiny]
        1\ar[drr,dash,dashed,green]& \hspace{1cm}&
            \\
        & &\theta\ar[ull, bend right=30]\ar[ull, bend right=50]\ar[ull, "TW", bend right=70]\ar[dll, bend left=30]\ar[dll, bend left=50]\ar[dll, "TW", bend left=70]\ar[dll,squiggly,dash,red]\\
        0\ar[uu, "TW", bend left=30]& &
\end{tikzcd}
\caption{Connectivity map for $F(1)>0$. In red, it is shown an admissible continuous path of steady-states (for any $\Omega$ and any $\mu>0$) connecting stationary solutions. In green, it is an admissible and continuous path of steady-states connecting two stationary solutions, but in this case, its existence depends on $\Omega$ and $\mu$. In black, traveling waves for the Cauchy problem are shown. The Traveling wave from $w\equiv0$ to $w\equiv1$ is unique while the traveling waves from $w\equiv\theta$ to $w\equiv1$ or to $w\equiv0$ are infinitely many.}\label{maptetasmall}
\end{figure}
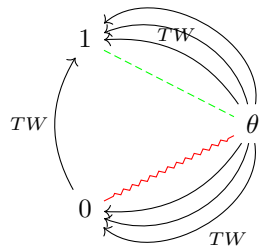

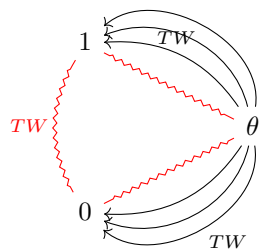
\begin{figure}[H]
\centering
\begin{tikzcd}[column sep=tiny]
        1\ar[drr,squiggly,dash,red]& \hspace{1cm}&
            \\
        & &\theta\ar[ull, bend right=30]\ar[ull, bend right=50]\ar[ull, "TW", bend right=70]\ar[dll, bend left=30]\ar[dll, bend left=50]\ar[dll, "TW", bend left=70]\ar[dll,squiggly,dash,red]\\
        0\ar[uu,squiggly,dash, "TW", red, bend left=30]& &
\end{tikzcd}
\caption{Connectivity map for $F(1)=0$. In red, an admissible continuous path of steady-states (for any $\Omega$ and any $\mu>0$) connecting stationary solutions is shown. The traveling wave from $w\equiv0$ to $w\equiv1$ is unique  and stationary, giving a continuous path of admissible steady-states connecting $w\equiv0$ and $w\equiv1$. In black, non-stationary traveling waves for the Cauchy problem are shown.  The traveling waves from $w\equiv\theta$ to $w\equiv1$ or to $w\equiv0$ are infinitely many.}\label{mapteta12}
\end{figure}

Further perspectives and problems are:
\begin{itemize}
 \item To study the structure of models carrying spatial heterogeneity.
 More realistic models carry more spatial dependences.
 \begin{equation*}
 \begin{dcases}
    u_t-\text{div}\left(A(x)\nabla u\right)+\langle b(x), \nabla u\rangle=f(u,x) &\qquad (x,t)\in \Omega\times(0,T],\\
  u=a(x,t) &\qquad (x,t)\in\partial \Omega\times (0,T],\\
  0\leq u(x,0)\leq 1.
 \end{dcases}
 \end{equation*}
 For instance, the carrying capacity or the diffusion can vary 
 depending on the space or having space-dependent drifts. This is tackled in \cite{drift} for the case of spatially heterogeneous drifts,
  \begin{equation*}
 \begin{dcases}
    u_t-\Delta u+\left\langle \frac{\nabla N(x)}{N(x)}, \nabla u\right\rangle=f(u) &\qquad (x,t)\in \Omega\times(0,T],\\
  u=a(x,t) &\qquad (x,t)\in\partial \Omega\times (0,T],\\
  0\leq u(x,0)\leq 1.
 \end{dcases}
 \end{equation*}
 where the authors extend the results of the present work. The presence of heterogeneity leads, for example, can lead to obstructions for reaching $w\equiv 1$. Furthermore, an extended version of the staircase method \cite{DARIO} is needed for controlling towards $w\equiv \theta$ for a small varying heterogeneity on the drift.
 
 \item The proof given here is based on the staircase method \cite{DARIO}, this strategy requires, by construction, a long time for achieving the target. The construction of the path guarantees that for specific initial data, the set of controls that drive from $u_0$ to the target $w\equiv \theta$ is not empty for $T$ big. But this geometrical construction does not give us any insight into controls that are not close to the path. For example, how can the dynamical control associated with the minimal controllability time be?
 \item Following the previous point, many perspectives are open, for instance, can we build a control that sends our state to the stable manifold of $w\equiv \theta$? How much time do we need to reach this manifold? 
 
 \item Other nonlinearities can also be considered, affecting, for instance, the principal part, such as the $p$-Laplacian. In these cases, the phase-plane analysis will be more intricate.
 
 \item To extend these results to systems of several coupled semilinear PDEs. More realistic ecological models or chemical reactions will carry a higher number of species with different relationships in the nonlinearity \cite{PERTHAME,MURRAY-BIO}. For example

 \begin{equation*}
  \begin{dcases}
   \partial_tu_1-\mu_1\Delta u_1=f_1(u_1,u_2,u_3)&\quad (x,t)\in\Omega\times(0,T],\\
    \partial_tu_2-\mu_2\Delta u_2=f_2(u_1,u_2,u_3)&\quad (x,t)\in\Omega\times(0,T],\\
   \partial_tu_3-\mu_3\Delta u_3=f_3(u_1,u_2,u_3)&\quad (x,t)\in\Omega\times(0,T],\\
   u_1=a(x,t)\in[0,1]&\quad (x,t)\in\partial \Omega\times(0,T],\\
   \frac{\partial}{\partial \nu} u_j=0&\quad (x,t)\in\partial \Omega\times(0,T]\quad j=2,3,\\
   0\leq u_i(x,0)\leq 1&\quad i=1,2,3.
  \end{dcases}
 \end{equation*}
 \item Due to technical reasons regarding the construction of the path, we have considered a control acting in the whole boundary. An important future perspective is to construct the path taking only a control in a part of the boundary $\eta \subset \partial\Omega$:
\begin{equation*}
 \begin{dcases}
  u_t-\mu\Delta u=f(u)&\quad (x,t)\in\Omega\times(0,T],\\
  u(x,t)=a(x,t)&\quad (x,t)\in\eta\times(0,T],\\
    \frac{\partial}{\partial \nu}u(x,t)= 0&\quad (x,t)\in\partial\Omega\backslash \eta\times(0,T],\\
  0\leq u(x,t)\leq 1&\quad (x,t)\in\Omega\times[0,T].
 \end{dcases}
\end{equation*}

\end{itemize}

\section*{Acknowledgments}
We thank Idriss Mazari for his valuable comments and José Ramón Uriarte for giving the first motivation of the problem in the context of evolutionary game theory in linguistics.
\section*{Funding}
This project has received funding from the European Research Council (ERC) under the European Union's Horizon 2020 research and innovation programme (grant agreement NO. 694126-DyCon). The work of both authors was partially supported by the Grant MTM2017-92996-C2-1-R COSNET of MINECO (Spain) and by the Air Force Office of Scientific Research (AFOSR) under Award NO. FA9550-18-1-0242. The work of E.Z. was partially funded by the Alexander von Humboldt-Professorship program, the European Unions Horizon 2020 research and innovation programme under the Marie Sklodowska-Curie grant agreement No.765579-ConFlex, the Grant ICON-ANR-16-ACHN-0014 of the French ANR and the Transregio 154 Project ``Mathematical Modelling, Simulation and Optimization Using the Example of Gas Networks'' of the German DFG

\appendix
\

\section{Estimates on the thresholds}\label{appendixMONO}

The estimates on the thresholds can be found in the classical literature \cite{BERESTYCKI-LIONS,PLLEPSSEQ}.  Here we will provide the explicit computation of the lower bound following a variational approach. 

\begin{prop}[A lower bound for $\mu^*$]
 
  Assume that $f(0)=f(\theta)=f(1)=0$, and that $f'(0)<0$, $f'(1)<1$, $f'(\theta)>0$. Moreover consider $F(v)=\int_0^v f(s)ds$ and assume that $F(1)>0$.
  Consider $\Omega\subset\mathbb{R}^N$ be a bounded set with boundary $C^2$, consider also the following problem:
  \begin{equation}\label{upbound}
  \begin{dcases}
   -\mu\Delta u= f(u)\quad&x\in\Omega,\\
   u=0\quad&x\in\partial\Omega,\\
   u>0\quad &x\in\Omega.
  \end{dcases}
  \end{equation}
  Denote by $B_\Omega$ a ball of maximal measure inside $\Omega$, $B_\Omega\subset\Omega$.
  Then, for any $\mu>0$ fulfilling 
   \begin{equation*}
     \mu<\frac{2\delta^2\Gamma\left(\frac{N}{2}+1\right)^{2/N}\left(F(\theta)+(1-\delta)^N(F(1)-F(\theta))\right)m(B_\Omega)^{2/N}}{\pi \left(1-(1-\delta)^N\right)}
    \end{equation*}
  there exists a solution of the problem \eqref{upbound}, where $\delta>0$ fulfills
    \begin{equation*}
   \delta < 1-\left(\frac{-F(\theta)}{F(1)-F(\theta)}\right)^{1/N}
  \end{equation*}
  and $\Gamma$ is the gamma function.
  This implies a lower bound for $\mu^*$.
 \end{prop}
 
 \begin{remark}
  Roughly speaking, the proposition says that if there exists a ball  big enough inside the domain under consideration, then there is multiplicity of solutions.
 \end{remark}

 \begin{proof}
    We know that $w\equiv 0$ is a solution of the Euler-Lagrange equations of the corresponding functional.
    \begin{equation*}
     I[v]=\frac{1}{2}\int_\Omega |\nabla v|^2 -\frac{1}{\mu} \int_\Omega F(v)dx
    \end{equation*}

    We want to find a function such that $I[v]<0$. We consider a ball inside our domain $\Omega$, $B_\Omega$. Then, we construct a family of functions  $v_\delta\in H^1_0(B_\Omega)$.
    
    The idea is first ensure under which conditions on $\delta$ we have that $\int_\Omega F(v(x))>c>0$. Once we have this, we choose $\mu$ in order to dominate the term $
     \frac{1}{2}\int_\Omega |\nabla v|^2dx$,  that will depend only on the $\delta$ chosen before and thus we constructed $v$ such that $I[v]<0$.    

    We consider another ball inside the previous ball defined by $
     (1-\delta)B_\Omega:=\left\{x\in \mathbb{R}^N\text{ s.t. }\frac{x}{1-\delta}\in B_\Omega\right\}$,  for $1>\delta>0$.
    We define $v_\delta$ in the following way, let $R$ be the radius of $B_\Omega$, $(1-\delta)R$ will be the radius of $(1-\delta)B_\Omega$:

    \begin{equation*}
     v_\delta(x)=\begin{dcases}
        1\quad &\text{if }x\in (1-\delta)B_\Omega,\\
        -\frac{1}{\delta R}(\|x\|_2-R)\quad& \text{if }x\in B_\Omega \backslash(1-\delta)B_\Omega.\\
     \end{dcases}
    \end{equation*}

    Note that $v_\delta\in H_0^1(B_\Omega)$ and extending it to be zero in $\Omega\backslash B_\Omega$ we have a function in $H^1_0(\Omega)$.
    Then, we have that:
    \begin{equation*}
    |\nabla v_\delta|^2=\begin{dcases}
     0 \quad &\text{in } (1-\delta)B_\Omega,\\
     \frac{\pi}{\delta^2} \left( m(B_\Omega)\Gamma\left(\frac{N}{2}+1\right)\right)^{-2/N}\quad & \text{in }B_\Omega\backslash(1-\delta)B_\Omega,
    \end{dcases}
    \end{equation*}
        where $\Gamma$ denotes the gamma function, and the term $$\pi\left( m(B_\Omega)\Gamma\left(\frac{N}{2}+1\right)\right)^{-2/N}, $$ comes from the volume of a $N$ dimensional sphere $$m(B_\Omega)=\frac{\pi^{N/2}}{\Gamma\left(\frac{N}{2}+1\right)}R_{B_\Omega}^N.$$ Moreover we have that:
    \begin{align*}
     &m\left((1-\delta)B_\Omega\right)=(1-\delta)^Nm(B_\Omega),\\
     &m\left(B_\Omega\backslash(1-\delta)B_\Omega\right)=\left(1-(1-\delta)^N\right)m(B_\Omega),
    \end{align*}
    we want to find a pair $(\mu,\delta)$ for which $
     I[v]<0.$ 
      For doing so, first we choose $\delta>0$ to be small enough such that $\int_{B_\Omega} \int_0^{v(x)} f(s)ds dx>c>0$, we split the space integral in two parts:
    \begin{align*}
        \int_{B_\Omega} \int_0^{v(x)} f(s)ds dx&=\int_{B_\Omega\backslash (1-\delta) B_\Omega} \int_0^{v(x)} f(s)ds dx+\int_{(1-\delta)B_\Omega}\int_0^1f(s)ds dx\\
        &\geq \int_{B_\Omega\backslash (1-\delta) B_\Omega} F(\theta) dx+F(1)m\left((1-\delta)B_\Omega\right)\\
        &=  F(\theta) m\left(B_\Omega\backslash (1-\delta) B_\Omega\right) +F(1)m\left((1-\delta)B_\Omega\right)\\
        &=m(B_\Omega)\left[F(\theta)\left(1-(1-\delta)^N\right)+F(1)(1-\delta)^N\right]\\
        &=m(B_\Omega)\left[F(\theta)+(1-\delta)^N\left(F(1)-F(\theta)\right)\right]
    \end{align*}
    So, it will suffice if we ensure that $F(\theta)+(1-\delta)^N\left(F(1)-F(\theta)\right)>0$ which corresponds to ask that:
    \begin{equation}\label{anterior}
     \delta< 1-\left(\frac{-F(\theta)}{F(1)-F(\theta)}\right)^{1/N}
    \end{equation}
    We fix $\delta>0$ fulfilling \eqref{anterior} and now or goal is to choose $\mu$ small enough so that the space integral on $F(v(x))$ dominates the gradient part.
    \begin{align*}
     I[v_\delta]&=\int_\Omega \frac{1}{2} |\nabla v_\delta|^2-\frac{1}{\mu} F(v_\delta(x))dx\\
     &=\int_{B_\Omega} \frac{1}{2} |\nabla v_\delta|^2-\frac{1}{\mu} F(v_\delta(x))dx\\
     &\leq \int_{B_\Omega\backslash (1-\delta) B_\Omega} \frac{1}{2} |\nabla v_\delta|^2dx-\frac{1}{\mu} \left(F(\theta)+(1-\delta)^N\left(F(1)-F(\theta)\right)\right)m(B_\Omega)\\
     =&m(B_\Omega)\left(\frac{1}{2}\left((1-(1-\delta)^N)\right)\frac{\pi}{\delta^2}\left(m(B_\Omega)\Gamma\left(\frac{N}{2}+1\right)\right)^{-2/N}-\frac{1}{\mu} \left(F(\theta)+(1-\delta)^N\left(F(1)-F(\theta)\right)\right)\right)
    \end{align*}
    so, it will be sufficient if:
    \begin{equation*}
     \frac{1}{2}\left((1-(1-\delta)^N)\right)\frac{\pi}{\delta^2}\left(m(B_\Omega)\Gamma\left(\frac{N}{2}+1\right)\right)^{-2/N}-\frac{1}{\mu} \left(F(\theta)+(1-\delta)^N\left(F(1)-F(\theta)\right)\right)<0
    \end{equation*}
    which corresponds to:
    \begin{equation*}
     \mu<\frac{2\delta^2\Gamma\left(\frac{N}{2}+1\right)^{2/N}\left(F(\theta)+(1-\delta)^N(F(1)-F(\theta))\right)m(B_\Omega)^{2/N}}{\pi \left(1-(1-\delta)^N\right)}
    \end{equation*}

 \end{proof}
   \begin{remark}
   Notice that the structure of the proof of Proposition \ref{Bistable} also works for the monostable case. When bounding by above the integral of the primitive, we will have $F(1)$ instead of $F(1)-F(\theta)$ because the primitive in the monostable case is monotone
  \end{remark}

\bibliography{CBIB,url = false,URL=false}

\begin{thebibliography}{10}
\expandafter\ifx\csname url\endcsname\relax
  \def\url#1{\texttt{#1}}\fi
\expandafter\ifx\csname urlprefix\endcsname\relax\def\urlprefix{URL }\fi
\expandafter\ifx\csname href\endcsname\relax
  \def\href#1#2{#2} \def\path#1{#1}\fi

\bibitem{KOLMOGOROV37}
A.~Kolmogorov, {\'E}tude de l'{\'e}quation de la diffusion avec croissance de
  la quantit{\'e} de mati{\`e}re et son application {\`a} un probl{\`e}me
  biologique, Bull. Univ. Moskow, Ser. Internat., Sec. A 1 (1937) 1--25.

\bibitem{JEVANS}
J.~Evans, Nerve axon equations 4: the stable and unstable impulse, Indiana
  Univ. Math. J. 24~(12) (1975) 1169--1190.

\bibitem{PERTHAME}
B.~Perthame, Parabolic equations in biology : growth, reaction, movement and
  diffusion, Lecture notes on mathematical modelling in the life sciences,
  Springer, 2015.

\bibitem{HOFBAUER}
J.~Hofbauer, V.~Hutson, G.~Vickers, Travelling waves for games in economics and
  biology, Nonlinear Anal-Theor 30~(2) (1997) 1235 -- 1244.

\bibitem{Hutson2000}
V.~Hutson, K.~Mischaikow, G.~T. Vickers, Multiple travelling waves in
  evolutionary game dynamics, Jpn. J. Ind. Appl. Math 17~(3) (2000) 341.

\bibitem{DEMASI}
A.~De~Masi, P.~Ferrari, J.~Lebowitz, Reaction diffusion equations for
  interacting particle systems, J. Stat. Phys. 44~(3-4) (1986) 589--644.

\bibitem{VOGL}
K.~Prochazka, G.~Vogl, Quantifying the driving factors for language shift in a
  bilingual region, P. Natl. Acad. Sci. USA 114~(17) (2017) 4365--4369.

\bibitem{fursikov1996}
A.-V. Fursikov, O.-Y. Imanuvilov, Controllability of evolution equations,
  no.~34, Seoul National University, 1996.

\bibitem{FC-ZZ}
E.~Fern{\'{a}}ndez-Cara, E.~Zuazua, Null and approximate controllability for
  weakly blowing up semilinear heat equations, Ann I. H. Poincare-An. 17~(5)
  (2000) 583 -- 616.

\bibitem{lebeau1995}
G.~Lebeau, L.~Robbiano, Contr{\^o}le exact de l{\'e}quation de la chaleur,
  Commun. Part. Diff. Eq. 20~(1-2) (1995) 335--356.

\bibitem{POUCHOL}
C.~Pouchol, E.~Tr{\'{e}}lat, E.~Zuazua, Phase portrait control for 1d
  monostable and bistable reaction{\textendash}diffusion equations,
  Nonlinearity 32~(3) (2019) 884--909.

\bibitem{LionsMagenesII}
J.-L. Lions, E.~Magenes, Non-homogeneous boundary value problems and
  applications, Vol.~II, Springer-Verlag, 1972.

\bibitem{DARIO}
D.~Pighin, E.~Zuazua, Controllability under positivity constraints of
  semilinear heat equations, Math. Control Relat. F. 8 (2018) 935.

\bibitem{PROTTER2012maximum}
M.~Protter, H.~Weinberger, Maximum principles in differential equations,
  Springer Science \& Business Media, 2012.

\bibitem{FIFEcomparison}
P.~Fife, M.~Tang, Comparison principles for reaction-diffusion systems:
  irregular comparison functions and applications to questions of stability and
  speed of propagation of disturbances, J. Differ. Equations 40~(2) (1981)
  168--185.

\bibitem{DARIOWAVE}
D.~Pighin, E.~Zuazua, Controllability under positivity constraints of multi-d
  wave equations, in: Trends in Control Theory and Partial Differential
  Equations, Springer, 2019, pp. 195--232.

\bibitem{CORON}
J.-M. Coron, Control and Nonlinearity, American Mathematical Society, Boston,
  MA, USA, 2007.

\bibitem{PLLEPSSEQ}
P.-L. Lions, On the existence of positive solutions of semilinear elliptic
  equations, SIAM Rev. 24~(4) (1982) 441--467.

\bibitem{BERESTYCKI-FRENCH}
H.~Berestycky, Le nombre de solutions de certains probl\`{e}mes
  semi-lin\'{e}aires elliptiques, J. Funct. Anal. 40~(1) (1981) 1 -- 29.

\bibitem{LOHEAC}
J.~Loheac, E.~Trelat, E.~Zuazua, Minimal controllability time for the heat
  equation under unilateral state or control constraints, Math. Mod. Meth.
  Appl. S. 27~(09) (2017) 1587--1644.

\bibitem{drift}
I.~Mazari, D.~Ruiz-Balet, E.~Zuazua, Constrained control of bistable
  reaction-diffusion equations: Gene-flow and spatially heterogeneous models,
  preprint: https://hal.archives-ouvertes.fr/hal-02373668/document.

\bibitem{CORON-TRELAT}
J.-M. Coron, E.~Tr{\'e}lat, Global steady-state controllability of
  one-dimensional semilinear heat equations, SIAM J. Control. Optim. 43~(2)
  (2004) 549--569.

\bibitem{lions1984structure}
P.-L. Lions, Structure of the set of steady-state solutions and asymptotic
  behaviour of semilinear heat equations, J. Differ. Equations 53~(3) (1984)
  362--386.

\bibitem{CAZENAVE-HARAUX}
T.~Cazenave, A.~Haraux, et~al., An introduction to semilinear evolution
  equations, Vol.~13, Oxford University Press on Demand, 1998.

\bibitem{POLACIK1996}
P.~Pol{\'a}{\v{c}}ik, K.~Rybakowski, Nonconvergent bounded trajectories in
  semilinear heat equations, J. Differ. Equations 124~(2) (1996) 472--494.

\bibitem{POLACIK2002}
P.~Pol{\'a}{\v{c}}ik, F.~Simondon, Nonconvergent bounded solutions of
  semilinear heat equations on arbitrary domains, J. Differ. Equations 186~(2)
  (2002) 586--610.

\bibitem{SIMON83}
L.~Simon, Asymptotics for a class of non-linear evolution equations, with
  applications to geometric problems, Ann. Math. (1983) 525--571.

\bibitem{JENDOUBI1998}
M.~Jendoubi, A simple unified approach to some convergence theorems of l.
  simon, J. Funct Anal. 153~(1) (1998) 187--202.

\bibitem{HARAUXJENDOUBI}
A.~Haraux, M.~Jendoubi, The convergence problem for dissipative autonomous
  systems: Classical methods and recent advances, Springer, 2015.

\bibitem{MATANO1978}
H.~Matano, et~al., Convergence of solutions of one-dimensional semilinear
  parabolic equations, J. Math. Kyoto. U. 18~(2) (1978) 221--227.

\bibitem{HERVE}
H.~Le~Dret, Nonlinear Elliptic Partial Differential Equations: An Introduction,
  Universitext, Springer International Publishing, 2018.

\bibitem{FIFEBOOK}
P.~Fife, Mathematical Aspects Of Reacting And Diffusing Systems, Vol.~28,
  Springer Science \& Business Media, 1979.

\bibitem{evans10}
L.~Evans, Partial differential equations, American Mathematical Society,
  Providence, R.I., 2010.

\bibitem{CAZENAVE-SEMILINEARELL}
T.~Cazenave, An introduction to semilinear elliptic equations, Editora do
  Instituto de Matem{\'a}tica, Universidade Federal do Rio de Janeiro, Rio de
  Janeiro 164.

\bibitem{MURRAY-BIO}
J.~Murray, Mathematical Biology I. An Introduction, 3rd Edition, Vol.~17 of
  Interdisciplinary Applied Mathematics, Springer, New York, 2002.

\bibitem{BERESTYCKI-LIONS}
H.~Berestycky, P.-L. Lions, Some applications of the method of super and
  subsolutions, in: Bifurcation and Nonlinear Eigenvalue Problems, Lecture
  Notes in Mathematics 782,, Springer-Verlag, Berlin, 1980, pp. 16 -- 41.

\end{thebibliography}

\end{document}